\documentclass[pdflatex,sn-basic, Numbered]{sn-jnl}% Math and Physical Sciences Numbered Reference Style
\usepackage{graphicx}%
\usepackage{multirow}%
\usepackage{amsmath,amssymb,amsfonts}%
\usepackage{amsthm}%
\usepackage{mathrsfs}%
\usepackage[title]{appendix}%
\usepackage{xcolor}%
\usepackage{enumitem}%
\usepackage{textcomp}%
\usepackage{manyfoot}%
\usepackage{booktabs}%
\usepackage{algorithm}%
\usepackage{algorithmicx}%
\usepackage{algpseudocode}%
\usepackage{listings}
\usepackage{dsfont}
\usepackage{subcaption}
\theoremstyle{thmstyleone}%
\newtheorem{theorem}{Theorem}%  meant for continuous numbers
\newtheorem{proposition}[theorem]{Proposition}% 
\newtheorem{lemma}[theorem]{Lemma}
\theoremstyle{thmstyletwo}%

\theoremstyle{thmstylethree}%

\newcommand{\bs}[1]{\boldsymbol{#1}}

\begin{document}

\title[Article Title]{A Deferred Correction, Continuous Galerkin Method for Curvilinear Staggered-Grid Lagrangian Hydrodynamics} 

%%=============================================================%%
%% GivenName	-> \fnm{Joergen W.}
%% Particle	-> \spfx{van der} -> surname prefix
%% FamilyName	-> \sur{Ploeg}
%% Suffix	-> \sfx{IV}
%% \author*[1,2]{\fnm{Joergen W.} \spfx{van der} \sur{Ploeg} 
%%  \sfx{IV}}\email{iauthor@gmail.com}
%%=============================================================%%

\author*[1]{\fnm{Steven} \sur{Walton}}\email{stevenw@lanl.gov}

\author[2]{\fnm{Svetlana} \sur{Tokareva}}\email{tokareva@lanl.gov}
% \equalcont{These authors contributed equally to this work.}

\author[3]{\fnm{Nathaniel} \sur{Morgan}}\email{nmorgan@lanl.gov}
% \equalcont{These authors contributed equally to this work.}

\affil*[2,3,1]{\orgname{Los Alamos National Laboratory}, \orgaddress{\city{Los Alamos}, \postcode{87545}, \state{New Mexico}, \country{USA}}}

% \affil[2]{\orgdiv{Department}, \orgname{Organization}, \orgaddress{\street{Street}, \city{City}, \postcode{10587}, \state{State}, \country{Country}}}

% \affil[3]{\orgdiv{Department}, \orgname{Organization}, \orgaddress{\street{Street}, \city{City}, \postcode{610101}, \state{State}, \country{Country}}}

%%==================================%%
%% Sample for unstructured abstract %%
%%==================================%%

\abstract{We present a continuous Galerkin, Deferred Correction (cG-DeC) method for the equations of Lagrangian
hydrodynamics. The proposed scheme combines a high-order continuous finite element discretization
with an explicit Deferred Correction time integrator, yielding a formulation
that avoids the inversion of a global sparse mass matrix at each update. To stabilize the method in the
presence of strong shocks we introduce a modification of the hyperviscosity model that is compatible with
the cG-DeC framework and likewise does not require a global solve to construct the finite element
approximation to the polyharmonic operator. We further present an alternative reformulation of
the DeC iteration which admits a simple recursive algorithmic structure and provides insight
into previously observed convergence behavior of explicit DeC applied to hyperbolic partial differential equations.
Conservation of momentum and total energy of the resulting fully discrete scheme is analyzed. A set of numerical experiments
illustrates the accuracy and robustness of the cG-DeC scheme. Comparisons with a continuous Galerkin Runge-Kutta (cG-RK) method are provided.}

\keywords{Lagrangian hydrodynamics, finite elements, deferred correction}

\pacs[MSC Classification]{65M60, 76M10, 65M12}

\maketitle

\section{Introduction}\label{sec:Intro}

Recently, a careful review of historical records has shed new light
on the early development of Lagrangian hydrodynamics codes.
The principal figures behind this work were Skyrme, Peierls, von Neumann, and Richtmyer \cite{origins_Lag_Hydro}.
The foundations laid by these towering minds of the last century continue to inspire research in numerical methods today.  Owing to their accuracy and robustness, Lagrangian schemes remain the method of choice for a wide range of continuum mechanics
applications. In the setting of compressible flow, the original schemes of Skyrme, Richtmyer, and von Neumann took the form of
one-dimensional, staggered finite difference methods aimed almost exclusively at the simulation of shock dynamics.
This focus gave rise to influential innovations such as the concept of artificial viscosity, generally credited to the landmark
paper of von Neumann and Richtmyer \cite{vonNeumann.1950}. We refer the reader to
\cite{origins_Lag_Hydro, Mattsson.2015, Margolin2023, Margolin.2024.Richtmyer} for engaging discussions on the historical
development of these methods, including the perhaps lesser-known contributions of Peierls, Skyrme, and Richtmyer.

Without exaggeration, the field of numerical methods for hyperbolic conservation laws has expanded enormously since those
early days, with foundational contributions from Godunov \cite{Godunov.1959}, Shu \cite{Shu}, Titarev and Toro \cite{Titarev.2004},
and Cockburn and Shu \cite{CockburnShu}, to name only a few, usually in the context of finite volume or discontinuous Galerkin methods. High-order discontinuous Galerkin
formulations for Lagrangian schemes were proposed in \cite{Lieberman.2019, Liu.2018, Liu.2019, Liu.2022}. Significant advances were
made in continuous finite element methods for hyperbolic balance laws, as exemplified by 
\cite{selmin_formaggia_1996, Abgrall.2017, DeconinckHirsch1981, CsikDeconinck2002, Abgrall2006ENO, RicchiutoCsikDeconinck2005, RicchiutoBollermann2009, Abgrall2006CurrentStatus}.

It was the influential work of Dobrev et al. \cite{Dobrev.2012}, in which the authors introduced a continuous,
high-order finite element method for Lagrangian hydrodynamics, that demonstrated that high-order continuous Galerkin methods 
are an effective option in Lagrangian simulations. A continuous treatment of the momentum equation has the appealing
feature of removing the need for a complicated multi-dimensional Riemann solver while stably advancing the equation of motion. 
This advantage, however, comes at the cost of having to invert a global sparse matrix at each momentum update. In \cite{RD1D}, and
subsequently in the multi-dimensional case \cite{RD_SGH}, a staggered-grid Residual Distribution (RD) scheme was developed
using a continuous Galerkin approximation. The RD method employs a Deferred Correction (DeC)
time-integration procedure that allows high-order mass lumping and thereby avoids inverting a global mass matrix, 
but it requires a ``distribution step'' to achieve high-order, owing to a first-order viscosity term present in the residuals.

In the present work, we apply a continuous Galerkin, Deferred Correction (cG-DeC) method to the equations of Lagrangian hydrodynamics.
In the context of hyperbolic balance laws, the cG-DeC method has been employed in Eulerian coordinates for linear advection problems, the Euler equations \cite{micalizziThesis, MicalizziTorlo2024},
and the shallow water equations where a well-balanced, high-order stabilization scheme was derived \cite{Micalizzi.2024}.  The cG-DeC method avoids the distribution step present in
Residual Distribution schemes, but then requires a stabilization term capable of maintaining arbitrary high-order.  In the settings just mentioned, the 
Continuous Interior Penalty (CIP) stabilization \cite{Burman.2004} is one such high-order stabilization that is successfully employed, but is ineffective on strong shocks.  In \cite{Bello-Maldonado.2020}, a polyharmonic limiter was introduced to obtain high-order convergence with an activated tensor artificial viscosity resulting in the hyperviscosity model.
In the present work, we provide a modification of the hyperviscosity model which is suitable for the cG-DeC method, as it does not rely on the inversion of a global 
sparse mass matrix to construct the finite element approximation to the polyharmonic operator.

The DeC iteration presented in \cite{Abgrall.2017} provides a theoretical basis for a fixed
number of iterations to achieve a desired accuracy.  However, it has been observed that the theoretical number of corrections \cite{micalizziThesis} may not give the expected order of convergence.
In the present article, we present an alternative reformulation of the DeC iteration which sheds some light on this phenomenon.

The avoidance of global mass-matrix
inversion is a notable design feature. In fact, on modern exascale, heterogeneous architectures, the
Krylov solvers typically used to invert the mass matrix introduce communication-bound operations
such as global dot products and norm evaluations, whose synchronization cost can dominate solver
performance
\cite{Hoemmen.2010, Carson2015CommunicationAvoidingKrylov}. Discretizations that bypass the global solve altogether are therefore
an attractive design target for high-order finite element methods.

\subsubsection*{Outline}
The remainder of the article proceeds as follows.  In Section \ref{sec:GovEqns} we recall the equations of Lagrangian hydrodynamics.  In Section \ref{sec:FE}, we describe the finite element discretization and stabilization mechanism.  
In Section \ref{sec:DeC} we provide a short description of the explicit Deferred Correction (DeC) method, following the formalism of Abgrall \cite{Abgrall.2017}.
We then present a new formulation for the explicit DeC method which allows for a simple recursive algorithmic formulation that is easy to incorporate within existing computational physics codes and provides some insight on existing convergence issues of explicit DeC methods applied to partial
differential equations.  Finally, the new DeC formulation is applied to the equations of Lagrangian hydrodynamics and the fully discrete cG-DeC scheme is analyzed.
The conservation of momentum and total energy are considered. In Section \ref{sec:Tests}, we provide several numerical illustrations of the cG-DeC scheme and provide comparisons with existing staggered-grid finite element methods. To conclude, in Section \ref{sec:Conclusions} we discuss some future research directions.

\section{Governing Equations}\label{sec:GovEqns}
We assume a Lagrangian (co-moving) reference frame and will make use of the notation \( \frac{\mathrm{d}}{\mathrm{d}t} \) to denote the material derivative.
 Let $(t,\bs x)\in [0, T]\times\Omega_t\subseteq\mathbb R^+\times\mathbb{R}^d$, $T>0$ 
for $d\in\{1,2,3\}$. The hydrodynamic pressure will be denoted by \( p = p(t,\bs x)\) and the stress tensor for the system is then 
\( \bs \sigma = -p {\mathrm{I}} + \bs \sigma_a \) with ${\mathrm{I}}$ the $d\times d$ identity matrix and \( \bs \sigma_a \) an artificial viscosity to
be described in Section \ref{sec:FE} below.  The conservation of mass demands that
    \begin{equation}\label{eqn:density}
        \frac{1}{\rho} \frac{\mathrm{d} \rho}{\mathrm{d}t} = -\nabla \cdot \bs u,
    \end{equation} 
    where \( \bs u =\bs u(t,\bs x)\) denotes the fluid velocity and \( \rho=\rho(t,\bs x) \) is the density. 
    
    Newton's second law gives us 
    \begin{equation}
        \rho \frac{\mathrm{d} \bs u}{\mathrm{d}t} = \nabla \cdot \bs \sigma.
    \end{equation}
    From the First Law of Thermodynamics, we obtain the governing equation for the specific internal energy,
    \( e =e(t, \bs x)\), which is given by
    \begin{equation}
        \rho \frac{\mathrm{d} e}{\mathrm{d}t} =  \bs \sigma :\nabla \bs u.
    \end{equation}
    The above equations are complemented by an equation of motion
    \begin{equation}
        \frac{\mathrm{d} \bs x}{\mathrm{d}t} = \bs u.
    \end{equation}
    
    We assume that we are dealing with an ideal gas so that \( e = c_v T \), with \( c_v \) the (constant) specific heat capacity and the pressure is related to the internal energy by the ideal gas equation of state (EoS), $p = (\gamma-1)\rho c_v T$, with $\gamma$ the adiabatic constant.

To summarize, we have the following well-known system of equations in Lagrangian coordinates
\begin{subequations}\label{eqn:Euler}
\begin{align}
    \frac{\mathrm d \rho}{\mathrm d t} &= -\rho\nabla\cdot \bs{u}\label{eqn:mass}\\
    \rho \frac{\mathrm d \bs{u}}{\mathrm d t} &= \nabla \cdot \bs\sigma,\label{eqn:momentum}  \\
    \rho \frac{\mathrm d e}{\mathrm d t} &=  \bs \sigma:\nabla \bs{u},\label{eqn:internal_energy}  \\
    \frac{\mathrm d \bs x}{\mathrm d t} &= \bs u.\label{eqn:motion}
\end{align}
\end{subequations}

In the next section, we describe the finite element spatial discretization for the above system of equations.

\section{Finite Element Discretization}\label{sec:FE}
In this section we describe the finite element discretization used for the spatial approximation.  The reader familiar with the continuous Galerkin discretization described in Dobrev et al. 2012 \cite{Dobrev.2012} may skip to the next subsection on the artificial viscosity.

\subsection{Notation, Mesh Discretization and Spatial Operators} Let us define a sequence of spatial configurations $(\Omega_t)_{t\ge 0}$ with $\Omega_t \subset\mathbb R^d$ and $d\in\{1,2,3\}$ the number of spatial dimensions.   
Let $\mathcal K_t$ be a conforming triangulation of $\Omega_t$ such that $\Omega_{t} = \bigcup_{c\in I^\mathcal{K}} K^c_t$, where $I^{\mathcal{K}}=\{1,\ldots,\#\mathcal{K}_t\}$ is the index set labeling the element $K^c_t\in\mathcal{K}_t$.  Further, we 
assume $\partial\Omega_{0}$ has straight edges or faces for $d=2$ or $3$, respectively.  Let $\hat{K}_{\mathcal{V}}$ denote the reference element for the kinematic quantities and $\hat{K}_{\mathcal{E}}$ be the reference element for the thermodynamic quantities approximated with finite elements.  We assume that $\hat K^{\mathcal{V}}$ and $\hat K^{\mathcal{E}}$ are both either the unit segment, a quadrilateral, a triangle, a tetrahedron or a hexahedron.  Let ${\rm q}\geq 1$ be an integer and denote by 
$\mathcal P_{\rm q}(\hat K)$ the space of polynomials of degree at most ${\rm q}$ on the kinematic reference element so that $\mathcal P_{\rm q-1}(\hat K)$ is the space of polynomials of degree at most ${\rm q-1}$ on the thermodynamic reference element with degrees of freedom indexed by $\hat{I}^{\mathcal{V}}=\{1,\ldots, \hat{N}^{\mathcal{V}} \}$ and $\hat{I}^{\mathcal{E}}=\{1,\ldots, \hat{N}^{\mathcal{E}} \}$, respectively.  Note that $\hat K^{\mathcal{V}}$ and $\hat K^{\mathcal{E}}$ are the same geometric object differing only in how their finite element degrees of freedom are defined, thus there is no ambiguity when writing $\hat{\bs{x}}$ to denote a location in either of the reference elements.  Let $\hat{I}^Q=\{1,\ldots,\hat{N}^Q\}$ be an index set for the quadrature points, $\{\hat{\bs{x}}_Q\}_{Q\in \hat{I}^Q}$, in both $\hat K^{\mathcal{V}}$ and $\hat K^{\mathcal{E}}$, defined by a Gauss-Legendre rule of the same order in each reference element (e.g. $(2\rm q)^d$ for tensor-product reference elements). 
Then, let $\{\varphi_i(\hat{\bs{x}})\}_{i\in \hat I^{\mathcal{V}}}$ denote the basis for $\mathcal P_{\rm q}(\hat K)$  on $\hat{K}_{\mathcal{V}}$ and 
$\{\psi_i(\hat{\bs{x}})\}_{i\in \hat I^{\mathcal{E}}}$ the basis for $\mathcal{P}_{\rm q - 1}(\hat K)$ on $\hat{K}_{\mathcal{E}}$, i.e. the shape functions for the reference elements.  On tensor-product elements, we use Lagrange basis functions and on simplices we use Bernstein basis functions.

Now, let $I^\mathcal{V}=\{1,\ldots,N^\mathcal{V}\}$ 
and $I^\mathcal{E} = \{1,\ldots,N^\mathcal{E}\}$ denote global index sets for the degrees of freedom of the kinematic and thermodynamic objects, respectively, approximated by finite elements.  Let $\hat{I}^{\mathcal{V}}_K$ and $\hat{I}^{\mathcal{E}}_K$ denote the per-element degrees of freedom on element $K^c_t$, with $t$ fixed. If each element $K^c_t$ carries $\#\hat I^{\mathcal{V}}$ kinematic and $\#\hat I^{\mathcal{E}}$ thermodynamic degrees of freedom, then the reference element and element local index sets are identical.  The kinematic degrees of freedom follow a Gauss-Lobatto-Legendre (GLL) distribution and the thermodynamic degrees of freedom follow a Gauss-Legendre (GL) distribution. We emphasize that this does not mean we employ a reduced quadrature. For example, on a tensor-product grid we have $(\rm q + 1)^d$ GLL kinematic degrees of freedom, $\rm q^d$ GL thermodynamic degrees of freedom and $(2\rm q)^d$ GL quadrature points.  Thus, the global kinematic degrees of freedom are continuous across the mesh and the thermodynamic degrees of freedom are not shared across elements.  We further define local-to-global maps between the index sets, which are given by $\iota_{\mathcal{V}}:I^{\mathcal{V}}_K\times I^{\mathcal{K}}\to I^{\mathcal{V}}$, $\iota_{\mathcal{E}}:I^{\mathcal{E}}_K\times I^{\mathcal{K}}\to I^{\mathcal{E}}$, respectively. 

We will make use of the per-element diffeomorphism $\Pi^{c}_t(\hat{\bs x}\,):\hat K^{\mathcal{V}} \to K^c_t$, which for each fixed $t\geq 0$ maps the
reference coordinates $\hat{\bs x} \in \hat K^{\mathcal{V}}$ (and therefore $\hat K^{\mathcal{E}}$) onto the physical element $K^c_t$ for all $c\in I^{\mathcal{K}}$. Let $\{\bs{x}_i(t)\}_{i\in I^{\mathcal{V}}}$ denote nodes in the case Lagrange basis functions are used for $\{\varphi_{i}\}_{i\in\hat{I}^{\mathcal{V}}}$ or control points if the Bernstein basis is used for $\{\varphi_{i}\}_{i\in\hat{I}^{\mathcal{V}}}$ defined on the mesh at time $t$.  Then, $\Pi^{c}_t(\hat{\bs x}\,)$ is defined by $$ \Pi^{c}_t(\hat{\bs x}\,)=\sum_{i\in I^{\mathcal{V}}_K } \varphi_{i}(\hat{\bs{x}})\bs x_{\iota_{\mathcal{V}}(i,c)}(t), $$
so that the fluid position in element $K_t^c$ is $\bs x(t,\hat{\bs x}) = \Pi^c_t(\hat{\bs x})$ and $\bs u(t,\bs{x}(t,\hat{\bs x})) =\partial_t \Pi^c_t(\hat{\bs x})$ represents the fluid velocity with $\bs{u}_i(t) = \partial_t \bs{x}_i(t)$ for all $i\in I^{\mathcal{V}}$.  Similarly, the internal energy in an element $K^c_t$ is represented by $e(t, \bs{x}(t, \hat{\bs x})) = \sum_{i\in I^{\mathcal{E}}_K}\psi_i(\hat{\bs x})e_{\iota_{\mathcal{E}}(i,c)}(t)$, where $\{e_i(t)\}_{i\in I^{\mathcal{E}}}$ are the values of the internal energy at the degrees of freedom in the elements of $\Omega_t$ or the control points, depending on whether a Lagrange or Bernstein basis is used.  In what follows, we write $\Pi^c_t$ for $\Pi^c_t(\hat{\bs x})$ to simplify the notation.
  
We will make use of the finite dimensional subspaces
\begin{equation}\label{eqn:kinematic_subspace}
    \begin{aligned}
        \mathcal V_h =\Big\{ \bs u \in L^2(\Omega_t)^d\,\Big|\, \bs u|_{K^c_t}\circ(\Pi^c_t)^{-1}
        \in \mathcal P_{\rm q}(\hat K)^d, \,\forall K^c_t \in\mathcal K_t \Big\}\cap C^0(\Omega_t)\subset \Big[H^1(\Omega_{t})\Big]^d,
    \end{aligned}
\end{equation}
for the kinematic space and 
\begin{equation}\label{eqn:thermodynamic_subspace}
    \begin{aligned}
        \mathcal E_h =\Big\{ e \in L^2(\Omega_t)\, \Big|\, e|_{K^c_t}\circ(\Pi_t^c)^{-1}
        \in\mathcal P_{\rm q-1}(\hat K), \,\forall K_t^c\in\mathcal K_t \Big\}\subset L^2(\Omega_{t}),
    \end{aligned}
\end{equation}
for the thermodynamic space.

The Jacobian at time $t$ and coordinate ${\bs x}$ is the $d\times d$ matrix defined by 
$$J(t,\bs{x}(t,\hat{\bs x})) = \nabla_{\hat{\bs x}}\Pi_t^c.$$ 

To approximate equation \eqref{eqn:mass}, we use a point-wise mass conservation as in \cite{Dobrev.2012}
\begin{equation}\label{eqn:pw_mass}
    \rho(t,\bs x(t,\hat{\bs x})) \det J(t,\bs{x}(t,\hat{\bs x})) = \rho(0, \bs x(0,\hat{\bs x})) \det J(0,\bs{x}(0,\hat{\bs x})),
\end{equation}
which follows from the Reynolds transport theorem.  Therefore, the density is not in the discrete spaces $\mathcal V_h$ or $\mathcal E_h$ but rather
a function defined on the quadrature points, $\bs x_q(t,\hat{\bs x})$, in the $K^c_t\in\mathcal{K}_t$ at time $t$ fixed.  The sound speed $c_s(t,\bs{x}(t,\hat{\bs x}))=\sqrt{\gamma\frac{p(t,\bs{x}(\hat{\bs x}))}{\rho(t,\bs{x}(\hat{\bs x}))}}$ and pressure $p(t,\bs{x}(\hat{\bs x}))$ are also defined to be functions on the quadrature points of the elements rather than in any finite element space 
and evaluated via the EoS using the above density defined in \eqref{eqn:pw_mass} and interpolation of the internal energy at the desired spatial coordinate. 
Placing all the thermodynamic quantities, except for the internal energy, outside of the finite element space obviates the need for hourglass controllers as demonstrated 
in \cite{ellis2010.thesis}.

Now, let us define our spatial operators.  All computations are performed on the reference element, $\hat K$ throughout. The mass matrices are denoted by $\mathcal M_{\mathcal{V}}$ and $\mathcal M_{\mathcal{E}}$ for the kinematic and thermodynamic spaces, respectively.  Let $i^\prime,j^\prime\in I^{\mathcal{V}}_K$, $i,j\in I^{\mathcal{V}}$ such that $(i,j) = (\iota_{\mathcal{V}}(i^\prime,c),\;\iota_{\mathcal{V}}(j^\prime,c))$ for each $c \in I^K$ with $K^c_t\ni (i,j)$, i.e. $c$ indexes the cells containing $i$ and $j$.  Then, the $i,j$ entry of $\mathcal{M}_{\mathcal{V}}$ is 
\begin{equation*}
    \big(\mathcal{M}_{\mathcal{V}}\big)_{i,j}= \sum_{K^c_t\ni(i,j)}\sum_{q\in I^Q}\omega_q \rho(0,\bs{x}_q(0,\hat{\bs{x}}) )\det J(0,\bs{x}_q(0,\hat{\bs{x}}))\varphi_{i^\prime}(\hat{\bs{x}}_q) \varphi_{j^\prime}(\hat{\bs{x}}_q),
\end{equation*}
where the $\{ \omega_q \}_{q\in I^Q}$ are the Gauss-Legendre quadrature weights.  For the thermodynamic mass matrix, for $i=\iota_{\mathcal{E}}(i^\prime,c)$ and $j=\iota_{\mathcal{E}}(j^\prime,c)$ with $i,j \in I^\mathcal{E}$ and $i^\prime, j^\prime \in I^\mathcal{E}_K$, we have \begin{equation*}
    \big(\mathcal{M}_{\mathcal{E}}\big)_{i,j} = \sum_{q\in I^Q}\omega_q \rho(0,\bs{x}_q(0,\hat{\bs{x}}) )\det J(0,\bs{x}_q(0,\hat{\bs{x}}))\psi_{i^\prime}(\hat{\bs{x}}_q) \psi_{j^\prime}(\hat{\bs{x}}_q),
\end{equation*}

We have used the relationship defined in equation \eqref{eqn:pw_mass} to eliminate time-dependence of the density in the mass matrix computation \cite{Dobrev.2012}. This, together with the fact that the basis functions $\{\varphi_i(\hat{\bs x})\}_{i\in I^{\mathcal{V}}_K},\; \{\psi_i(\hat{\bs x})\}_{i\in I^{\mathcal{E}}_K}$ are constant in time, makes the mass matrices time independent and thus, they are assembled once at the beginning of the simulation.  

The row-sum lumped mass matrices are $\mathcal D_{\mathcal{V}} = {\rm{diag}}(\mathcal M_{\mathcal{V}}\mathds{1}_{\mathcal V})$ and 
$\mathcal D_{\mathcal{E}} = {\rm{diag}}(\mathcal M_{\mathcal{E}}\mathds{1}_{\mathcal E})$, where $\mathds{1}_{\mathcal{V}}$ and $\mathds{1}_{\mathcal{E}}$ are vectors of all ones with lengths
$\#I^{\mathcal{V}}$ and $\#I^{\mathcal{E}}$, respectively.  It is important that the row-sum lumped mass matrices do not produce a degenerate lumped mass approximation. 
Thus, we use Lagrange polynomials with Gauss-Lobatto-Legendre nodes on tensor product elements of any order and linear Lagrange elements or arbitrary order Bernstein elements on simplices, all of which avoid degenerate lumped masses.

Now, let $i=\iota_{\mathcal{V}}(i^\prime,c)$ for $i\in I^{\mathcal{V}}$, $i^\prime \in I^{\mathcal{V}}_{K}$ and $j=\iota_{\mathcal{E}}(j^\prime,c)$ for $j\in I^{\mathcal{E}}$, $j^\prime \in I^{\mathcal{E}}_{K}$ for $K^c_t\in \mathcal K_t$.  Then, the weak form of the internal energy right-hand side in equation \eqref{eqn:internal_energy} gives rise to the force tensor \cite{Dobrev.2012}, whose entry for the degree-of-freedom pair $(i,j)\in I^{\mathcal{V}}\times I^{\mathcal{E}}$ at dimension $r$ is given by
\begin{equation}
    F_{i,j,r} = \sum_{K^c_t\ni i}\sum_{q\in I^Q}  \bs\sigma_{r,\ell}(t,\bs{x}_q(t,\hat{\bs x})) J^{-T}_{\ell,s}(t,\bs{x}_q(t,\hat{\bs x})) \partial_{\hat{\bs x}_s} \varphi_{i^\prime}(\hat{\bs x}_q) \psi_{j^\prime}(\hat{\bs x}_q) \omega_q \det J(t,\bs{x}_q(\hat{\bs x})),
\end{equation}
where $J^{-T}(t,\bs{x}(t,\hat{\bs x}))$ is the inverse transpose of the Jacobian and the matching $\ell$ and $s$ indices imply summation over physical dimensions.  With the boundary conditions $\bs u \cdot \hat n \big|_{\partial\Omega} = 0$ and using the definition of $\bs\sigma(t, \bs x(t,\hat{\bs x}))$, we can write the right-hand side of the momentum equation \eqref{eqn:momentum} in matrix-vector form as $-F\mathds{1}_{\mathcal{E}}$ from the 
partition of unity property of the thermodynamic basis.

Thus, the principal spatial operators are defined. In the next subsection we describe the stabilization mechanism for the scheme.

\subsection{Artificial Viscosity}
To stabilize the scheme, we use a limited artificial viscosity, which allows for a high-order solution in smooth regions but effectively stabilizes the solution at a shock. The formulation follows closely the hyperviscosity approach detailed in \cite{Bello-Maldonado.2020},
with the caveat that we use a few terms of a Neumann series, where the number of iterations depends on the mesh, basis type and polynomial order, to approximate the polyharmonic operator. More specifically, we have
\begin{equation}\label{eqn:artificial_viscosity}
    \bs{\sigma}_a(t,\bs{x}(t,\hat{\bs{x}})) = \mu \nabla^s \bs{u}(t,\bs{x}(t,\hat{\bs{x}})),
\end{equation}
where the operator $\nabla^s$ denotes the symmetric gradient and $\mu$ is the limited viscosity coefficient which is a function of the quadrature points.  Here, $\mu = \min(\mu_{\rm{hyp}}, \mu_{\rm std})$ (using the notation found in \cite{Bello-Maldonado.2020}) where 
$$\mu_{\rm{std}}= \rho \ell(\bs{x}(t,\hat{\bs{x}})) \Big( \beta_{\rm{vort}}\beta_{\rm{comp}} c_s(t,\bs{x}(t,\hat{\bs{x}}))  + \beta_{\rm{quad}}\ell(\bs{x}(t,\hat{\bs{x}}))|\lambda| \Big),$$ 
is a common ``linear plus quadratic'' viscosity coefficient \cite{Campbell.2001, CARAMANA.2006, Lipnikov.2010, KOLEV20098336, CARAMANA.1998, Caramana_Art_Visc.1998}. The coefficient function $\beta_{\mathrm{vort}}$ is a measure of vorticity, $\beta_{\mathrm{comp}}$ is a measure of the compression \cite{Dobrev.2012} and $\beta_{\rm{quad}}$ is a user defined parameter acting on the quadratic term.  
The length scale is given $\ell(\bs x(t,\hat{\bs{x}})) = \ell_0 \frac{\|J^{-1}(0,\bs{x}(0,\hat{\bs{x}})) J(t,\bs{x}(t,\hat{\bs{x}})) \cdot \bs s\|}{\| \bs s\|}$, with the initial length scale divided by the kinematic polynomial order \cite{Dobrev.2012, Guermond.2016.EntropyViscosity}. For example, on a tensor-product grid one has $\ell_0 = \frac{1}{\rm q}| K^c_t |^{1/d}$ for an element $K^c_t\in \mathcal{K}_t$ in the initial configuration.
The pair $(\lambda, \bs s)$ is the minimum spectral data pair of $\nabla^s\bs u(t,\bs{x}(t,\hat{\bs{x}}))$.  If the initial configuration does not guarantee that all elements have the same volume with straight edges, see \cite{Guermond.2016.EntropyViscosity} for an alternative definition of the initial length scale.  See \cite{Guermond.2016.EntropyViscosity} and \cite{Dobrev.2012} for a detailed discussion on 
alternative choices for the length scale and the consequences of getting this wrong. We use
$$\mu_{\rm{hyp}}=\beta_{\rm{hyp}}\rho(t,\bs{x}(t,\hat{\bs{x}}))\ell(\bs{x}(t,\hat{\bs{x}}))^{2\mathrm z}{\Bigg|\Delta^{\mathrm z} \bigg(\ell_0^2 \frac{\det{J(t,\bs{x}(t,\hat{\bs{x}}))}}{\det{J(0,\bs{x}(0,\hat{\bs{x}}))}} \sqrt{\nabla^s\bs{u}:\nabla^s\bs{u}} \bigg) \Bigg|}, \quad \mathrm z = 1,2,\ldots$$
as the hyperviscosity limiter, where $\beta_{\rm{hyp}}$ is another user parameter and the $|\cdot|$ operator is understood to be subject to a smoothing operation as described in \cite{Bello-Maldonado.2020}.  In this work, we use the $\mathrm z = 1$ case exclusively.  In \cite{Bello-Maldonado.2020}, the approximate Laplacian is $\mathcal{M}^{-1}S \approx \Delta$, where $S$ is a finite element diffusion operator and $\mathcal{M}$ is a kinematic mass matrix with density identically one. Since the diffusion equation does not suffer from the same dispersion error as hyperbolic equations do (outlined in \cite{GUERMOND.dispersionError}), it is reasonable to use collocated mass lumping on tensor product meshes to obtain a high-order approximation to  $\mathcal{M}^{-1}S$ or to use a few iterations of a Neumann series on simplex meshes. The Neumann series approximation to $\mathcal{M}^{-1}$, with $\mathcal{D}$ its row-sum lumped mass, is given by
$$  
    \mathcal{M}^{-1} \approx \sum_{\rm r = 0}^{\rm s}(I-\mathcal{D}^{-1}\mathcal{M})^{\rm r}\mathcal{D}^{-1},
$$
where $\rm s=0$ corresponds to an application of the lumped mass matrix and in the limit as $\rm s \rightarrow \infty$ we have equality \cite{GUERMOND.dispersionError}.  Then, our approximation for the finite element Laplace operator is 
$$
    \Delta \approx\sum_{\rm r = 0}^{\rm s}(I-\mathcal{D}^{-1}\mathcal{M})^{\rm r}\mathcal{D}^{-1}S,
$$
for the high-order limiting coefficient.  For the numerical tests in Section \ref{sec:Tests}, a single iteration of the above approximation was enough for our purposes. 

\subsection{The Semi-Discrete System}
 With the previously defined operators in hand, we may finally write the semi-discrete equations in matrix-vector form using the Picard integral formulation
\begin{subequations}\label{eqn:EulerSemiDisc}
\begin{align}
    \mathcal{M}_{\mathcal{V}}\big(\bs{u}(t)-\bs{u}(t_0)\big) &= -\int_{t_0}^{t}F(s)\mathds{1}_{\mathcal E}\mathrm d s,\label{eqn:FEmomentum}  \\
    \mathcal{M}_{\mathcal{E}}\big(e(t)- e(t_0)\big) &=  \int_{t_0}^{t}F(s)^T \bs{u}(s)\mathrm d s,\label{eqn:FEinternal_energy}  \\
    \bs x(t) - \bs{x}(t_0) &= \int_{t_0}^{t}\bs u(s)\mathrm d s.\label{eqn:FEmotion}
\end{align}
\end{subequations}
These equations \eqref{eqn:EulerSemiDisc} together with the pointwise mass conservation \eqref{eqn:pw_mass}, complete the spatial discretization.  We now move on to the description of Deferred Correction time integration to arrive at the fully discrete system.

\section{Deferred Correction}\label{sec:DeC}
The Deferred Correction method was introduced by Fox in 1947 \cite{Fox.1947} and later extended 
by Pereyra in 1966 \cite{Pereyra.1966}. For a detailed account of the history and classification
 of Deferred Correction schemes, see \cite{Ong.Spiteri.2020}. The Spectral Deferred Correction 
 method was introduced by Dutt, Greengard, and Rokhlin \cite{Dutt.2000}. In this formulation, 
 the ordinary differential equation is recast in integral form, and the resulting Picard integral 
 equation is approximated using a spectral quadrature rule. The correction process may therefore 
 be viewed as an iterative Picard-type refinement: each sweep improves the approximation by reducing 
 the defect in the discretized integral equation. Spectral Deferred Correction was extended to the 
 semi-implicit setting by Minion in 2003 \cite{Minion.2003}, and its application to splitting methods
  was investigated by Hagstrom and Zhou in 2006 \cite{Hagstrom.2006}.

Abgrall extended the above ideas to hyperbolic partial differential equations, providing a formalism which allowed for 
continuous finite element discretizations but avoiding mass matrix inversion \cite{Abgrall.2017}.  In the last decade, these
ideas have been successfully extended and applied in various settings and we refer the reader to the works \cite{Abgrall.2019, micalizziThesis, MicalizziTorlo2024, Michel.2021, Michel.2023}
and the references therein.

Let us describe the Deferred Correction methodology following \cite{Abgrall.2017}.  For any time interval $I_n=[t_n,t_{n+1}]$, we have $\rm{M}$ sub-time intervals defined by $\rm{M}+1$ sub-time nodes
$$
t_n =t_{0,n} < t_{1,n} < \ldots < t_{\rm m,n} < \ldots < t_{\rm M-1,n} < t_{\rm M,n} = t_{n+1},
$$
so that $\Delta t_{\rm m,n} = \xi_{\rm m,n} \Delta t$ with $0 = \xi_0 < \xi_1 < \ldots < \xi_{\rm{M}} = 1$ and $\Delta t = t_{n+1}-t_n$ is the increment length.  Throughout, we use $\rm m\in\{ 0, 1, \ldots ,\rm M\}$ to index the sub-time nodes for a given DeC discretization.

Now, to illustrate the idea, let us assume that we have discretized a scalar hyperbolic balance law using
the continuous Galerkin method and have the following semi-discrete matrix equation in Picard integral form
\begin{equation}\label{eqn:integral_eqn}
    \mathcal{M}\big(u(t_{n+1}) - u(t_n)\big) = \int_{t_n}^{t_{n+1}}f(u(t))\mathrm d t,
\end{equation}
where $\mathcal M\in \mathbb{R}^{\rm N \times \rm N}$ is the global, sparse mass matrix and $f(u(t))\in \mathbb{R}^{\rm N}$ is the resulting vector obtained from the Galerkin discretization which depends linearly or nonlinearly
on $u(t)$.  Here the integer $\rm N>0$ is the number of degrees of freedom obtained from the Galerkin discretization. We assume that $f$ is Lipschitz in $u$.  The iterations (also referred to as corrections) are indexed by $\rm k \in\{0,1,\ldots,\rm K\}$, so that
the solution at the $\rm m$-th sub-time node and $\rm k$-th correction in interval $I_n$ is denoted by $u^{(\rm k)}_{\rm m,n}$. Here, we take the zeroth iterate across $I_n$ to be initially constant and given by $u_n\approx u(t_n)$ so that
$u^{(0)}_{\rm m,n} = u_n$ for $1 \leq \rm m\leq \rm M$ and we do not correct the solution at sub-time node $t_{0,n}$ so that $u^{(\rm k)}_{0,n}=u_n$ for $1 \leq \rm k\leq \rm K$.  We note that other choices are possible, for example the zeroth iterate could be obtained from a cheap forward Euler solve at each sub-time node.  

A direct application of SDC would require inversion of the mass matrix to form an ordinary differential equation (ODE) for the error and correcting the error via the SDC iteration until we arrive
at $u_{n+1} = u^{(\rm K)}_{\rm M,n}$, where $ u_{n+1} \approx u(t_{n+1})$ is the solution at the last sub-time node in the interval $I_n$ (see \cite{Ong.Spiteri.2020}).  Instead of inverting $\mathcal M$, we can define the operator
\begin{equation}\label{eqn:L2}
    \mathcal L_2 u^{(\rm k)}_{\rm m,n} = \mathcal M(u^{(\rm k)}_{\rm m,n}-u_n) - \Delta t \sum_{\ell=0}^{\rm M} \Theta_{\rm m,\ell} f(u^{(\rm k)}_{\ell,n}),
\end{equation}
so that $\mathcal L_2 \big(u^*_{\rm m, n}\big) = \boldsymbol{0}$ is the finite element solution to equation \eqref{eqn:integral_eqn} above at time node $t_{\rm m,n}$ in interval $I_n$.    Here, we use a Gauss-Lobatto quadrature to define the sub-time node distribution in the interval $I_n$, but other choices are possible \cite{Abgrall.2019, Abgrall.2017}.  The matrix $\Theta$ is zero in the first row and contains the quadrature weights in each subsequent row to integrate the sub-time interval from $t_n$ to $t_{\rm m,n}$ for each $\rm m$.  However, directly solving for $u^*_{\rm m, n}$ for each $\rm m$ in \eqref{eqn:L2} does not provide much in the 
way of computational savings, and in fact amounts to the implicit Lobatto IIIA method \cite{hairer2006geometric}.  As an example, consider three sub-time nodes. Then, we have $\rm M = 2$ and with $t_{0,n} = 0$, $t_{1,n} = \frac12$ and $t_{\rm 2,n} = 1$ and the integration weight matrix
\[
    \Theta =
    \begin{pmatrix}
    0 & 0 & 0 \\
    \frac{5}{24} & \frac{1}{3} & -\frac{1}{24} \\
    \frac{1}{6} & \frac{2}{3} & \frac{1}{6}
    \end{pmatrix},
\]
corresponding to the 4th order Lobatto IIIA integrator \cite{hairer2006geometric}.  
To obtain an explicit method without inverting $\mathcal{M}$, the key idea is to define an operator $\mathcal L_1$ that is close enough to $\mathcal L_2$ which is
also cheap to invert. To this end, we have the following lemma due to Abgrall \cite{Abgrall.2017}.
\begin{lemma}\label{lemma:DeC}
        If two operators $\mathcal L_1,$ $\mathcal{L}_2$ depending on a parameter $\delta$ satisfy the following assumptions
        \begin{enumerate}
            \item There exists a unique solution $U^*$ to $\mathcal {L}_2$ such that $\mathcal L_2 U^* = \boldsymbol 0$;
            \item There exists a constant $C_0>0$ independent of $\delta$ such that for any $U$ and $V$, $\mathcal L_1$ satisfies the coercivity property
            $$C_0 \| U - V\| \leq \| \mathcal{L}_1 U - \mathcal L_1 V  \|; $$
            \item There exists a constant $C_1 \in (0, C_0)$ independent of $\delta$ such that for any $U$, $V$ the operator difference $\mathcal L_1 - \mathcal L_2$ satisfies the Lipschitz condition
            $$
            \| \big(\mathcal L_1 U - \mathcal L_2 U \big) - \big(\mathcal L_1 V - \mathcal L_2 V \big)\| \leq C_1\delta \| U- V\|;
            $$
        \end{enumerate}
        then if $\frac{C_1\delta}{C_0} < 1$ the iteration
        $\mathcal{L}_1(U^{(\rm k+1)}) = \mathcal L_1 U^{(\rm k)}  - \mathcal L_2 U^{(\rm k)} $ is convergent and the error after $\rm K$ corrections is smaller than $\big( \frac{C_1\delta}{C_0} \big)^{\rm K}\| U^{(0)} - U^*\|$.
    \end{lemma}
\begin{proof}
  The details on the choice of norm and proofs are given in \cite{Abgrall.2017}.
\end{proof}
We take our low-order operator to be,
\begin{equation}\label{eqn:L1}
    \mathcal L_1 u^{(\rm k)}_{\rm m ,n} = \mathcal{D}(u^{(\rm k)}_{\rm m,n}-u_n) - \Delta t_{\rm m , n} f(u_n),
\end{equation}
which satisfies the conditions of the above lemma.  Then, the usual DeC iteration from \cite{Abgrall.2017} is 
\begin{equation}\label{eqn:AbgrallDeC}
  \mathcal L_1 u^{(\rm k+1)}_{\rm m,n} = \mathcal L_1 u^{(\rm k)}_{\rm m,n} - \mathcal L_2 u^{(\rm k)}_{\rm m,n}.
\end{equation}
The iteration \eqref{eqn:AbgrallDeC} has been used successfully in many contexts (\cite{lohmann2019, kuzmin_hajduk, micalizziThesis, MicalizziTorlo2024, Abgrall.2019, Abgrall.2022} to name a few) and Lemma \ref{lemma:DeC} provides an estimate on the number of corrections needed to obtain a given error, where $\delta \sim \mathcal{O}(\Delta t) + \mathcal{O}( h)$ in Lemma \ref{lemma:DeC} above derived in \cite{Abgrall.2017} using an ${H}^{-1}$ dual norm to obtain the estimates for the operators $\mathcal{L}_1,\;\mathcal{L}_2$ acting on functions in a subspace of ${H}^1$.

We remark that the Lemma \ref{lemma:DeC} above is presented as it is usually found in the literature (see \cite{Abgrall.2017,Abgrall.2019} for example), with the vectors $U,V,U^*$ being vectors over the set of sub-time nodes for a fixed spatial degree of freedom.  
In the present work, it is convenient to consider how the iteration \eqref{eqn:AbgrallDeC} using the $\mathcal{L}_1$, $\mathcal{L}_2$ operators modify a state vector of finite element spatial degrees of freedom at fixed time nodes, $u^{(k)}_{m,n}$, as in equations \eqref{eqn:L1} and \eqref{eqn:L2} above.  This perspective is more appropriate when considering the global finite element operators, as will be seen below.

We have the following useful proposition which shows that the above iteration is equivalent to a simpler closed formula.
\begin{proposition}
  The DeC update \eqref{eqn:AbgrallDeC} is equivalent to the following iteration
  \begin{equation}\label{eqn:NeumannDeC}
    u^{(\rm k+1)}_{\rm m,n} = u_n + \Delta t \sum_{\ell=0}^{\rm M} \Theta_{\rm m,\ell}\sum_{\rm p=0}^{\rm k} A^{\rm p}\mathcal{D}^{-1} f(u^{(\rm k-\rm p)}_{\ell,n}),
  \end{equation}
  where $A = ( I - \mathcal{D}^{-1}\mathcal{M})$ and $I$ is the identity matrix that is compatible with the mass matrix.
\end{proposition}
\begin{proof}
This follows from a simple argument by induction.  Using the assumption $u^{(0)}_{m,n} = u^{(\rm k)}_{0,n} = u_n$, for $\rm k, \rm m >0$ 
and observing that the $\Delta t_{\rm m,n} f(u_n)$ term always cancels when computing 
$\mathcal L_1 u^{(\rm k+1)}_{\rm m,n} - \mathcal L_1u^{(\rm k)}_{\rm m,n}$ from our choice of $\mathcal L_1$, the first two corrections are 
$$
  u^{(1)}_{\rm m,n} = u_n + \Delta t \sum_{\ell=0}^{\rm M} \Theta_{\rm m,\ell}\mathcal D^{-1} f(u^{(0)}_{\ell,n}),
$$
and 
$$
  u^{(2)}_{\rm m,n} = u^{(1)}_{\rm m,n} - \mathcal D^{-1}\Big(\mathcal M (u^{(1)}_{\rm m,n} - u_n) -  \Delta t \sum_{\ell=0}^{\rm M} \Theta_{\rm m,\ell}f(u^{(1)}_{\ell,n})\Big).
$$
Substituting the expression for $u^{(1)}_{\rm m,n}$ into the RHS above, we get 

$$
  u^{(2)}_{\rm m,n} = u_n + \Delta t \sum_{\ell=0}^{\rm M} \Theta_{\rm m,\ell}\mathcal D^{-1} f(u^{(0)}_{\ell,n}) - \Delta t \sum_{\ell=0}^{\rm M} \Theta_{\rm m,\ell}\mathcal D^{-1}\mathcal M \mathcal D^{-1} f(u^{(0)}_{\ell,n}) +  \Delta t \sum_{\ell=0}^{\rm M} \Theta_{\rm m,\ell}\mathcal D^{-1}f(u^{(1)}_{\ell,n}).
$$
Rearranging
\begin{equation}
  \begin{aligned}
  u^{(2)}_{\rm m,n} = u_n + \Delta t \sum_{\ell=0}^{\rm M} \Theta_{\rm m,\ell}(\mathcal D^{-1} - \mathcal D^{-1} \mathcal M \mathcal D^{-1}) f(u^{(0)}_{\ell,n}) +  \Delta t \sum_{\ell=0}^{\rm M} \Theta_{\rm m,\ell}\mathcal D^{-1}f(u^{(1)}_{\ell,n})\\
                = u_n + \Delta t \sum_{\ell=0}^{\rm M} \Theta_{\rm m,\ell}(I-\mathcal D^{-1}\mathcal M)\mathcal D^{-1}f(u^{(0)}_{\ell,n}) + \Delta t \sum_{\ell=0}^{\rm M} \Theta_{\rm m,\ell}\mathcal D^{-1}f(u^{(1)}_{\ell,n}) \\
                = u_n + \Delta t \sum_{\ell=0}^{\rm M} \Theta_{\rm m,\ell}\sum_{\rm p=0}^1 (I-\mathcal D^{-1}\mathcal M)^{\rm p} \mathcal D^{-1} f(u^{(1-\rm p)}_{\ell,n}).
  \end{aligned}
\end{equation}

We assume the formula holds for the $\rm k$th correction and compute the $(\rm k+1)$th correction
$$
 u^{(\rm k+1)}_{\rm m,n} = u^{(\rm k)}_{\rm m,n} - \mathcal D^{-1}\Big(\mathcal M (u^{(\rm k)}_{\rm m,n} - u_n) - \Delta t\sum_{\ell=0}^{\rm M}\Theta_{\rm m,\ell} f(u^{(\rm k)}_{\ell,n}) \Big), 
$$
then use the inductive hypothesis to get
\begin{equation*}
  \begin{aligned}
  u^{(\rm k+1)}_{\rm m,n} = u_n + \Delta t \sum_{\ell=0}^{\rm M} \Theta_{\rm m,\ell}\sum_{p=0}^{\rm k-1}(I-\mathcal D^{-1}\mathcal{M})^{\rm p}\mathcal D^{-1} f(u^{(\rm k-1-\rm p)}_{\ell,n})\\
   + \Delta t \sum_{\ell=0}^{\rm M} \Theta_{\rm m,\ell}\mathcal D^{-1}\Big(\mathcal M \Big(\sum_{\rm p=0}^{\rm k-1}(I-\mathcal D^{-1}\mathcal{M})^{\rm p}\mathcal D^{-1} f(u^{(\rm k-1-\rm p)}_{\ell,n})\Big) + f(u^{(\rm k)}_{\ell,n}) \Big)\\
                  = u_n + \Delta t \sum_{\ell=0}^{\rm M} \Theta_{\rm m,\ell} \Big(( I - \mathcal D^{-1}\mathcal M)\Big(\sum_{\rm p=0}^{\rm k-1}(I-\mathcal D^{-1}\mathcal{M})^{\rm p}\mathcal D^{-1} f(u^{(\rm k-1-\rm p)}_{\ell,n})\Big) + \mathcal D^{-1}f(u^{(\rm k)}_{\ell,n})\Big)\\
                  = u_n + \Delta t \sum_{\ell=0}^{\rm M} \Theta_{\rm m,\ell} \Big( \sum_{\rm p=1}^{\rm k}(I-\mathcal D^{-1}\mathcal{M})^{\rm p}\mathcal D^{-1} f(u^{(\rm k-\rm p)}_{\ell,n}) + \mathcal D^{-1}f(u^{(\rm k)}_{\ell,n})\Big) \\
                  = u_n + \Delta t \sum_{\ell=0}^{\rm M} \Theta_{\rm m,\ell} \sum_{\rm p=0}^{\rm k}(I-\mathcal D^{-1}\mathcal{M})^{\rm p}\mathcal D^{-1} f(u^{(\rm k-\rm p)}_{\ell,n}),
\end{aligned}
\end{equation*}
which completes the proof.
\end{proof}
   We see from \eqref{eqn:NeumannDeC} that the DeC iteration constructs a Neumann series approximation to the mass matrix inverse with the caveat that the terms in the Neumann series act on the right-hand side at different correction levels.
 The convergence of the Neumann series depends on the spectral radius of $A$. It has been observed that the DeC iteration may not converge using the theoretical number of corrections (see \cite{micalizziThesis} for example), especially for high order Bernstein polynomials for which $A$ has a spectral radius approaching one as the polynomial order is increased.  We would like to point out that this does not contradict the estimate in Lemma \ref{lemma:DeC}, but suggests perhaps a sharper bound exists which incorporates the dependence on $A$ (or equivalent). 
 
 The iteration \eqref{eqn:NeumannDeC} above suggests a path to accelerate the convergence through preconditioning the $\mathcal L_2$ operators to reduce the spectral radius of $A$ and speed-up convergence of the cG-DeC method.  More specifically, we conjecture that it is possible to define a different $\mathcal{L}_1$ operator which is invertible but preconditions $\mathcal{M}$ better than $\mathcal{D}^{-1}$ for a given polynomial basis when applied in the iteration \eqref{eqn:NeumannDeC}.  The development of such a preconditioning strategy is beyond the scope of the present article.
  
\subsection{Application to the Equations of Lagrangian Hydrodynamics}\label{sec:DeCLagHydro}

Applying the DeC iteration \eqref{eqn:NeumannDeC} to the semi-discrete equations \eqref{eqn:FEmomentum}, \eqref{eqn:FEinternal_energy} and \eqref{eqn:FEmotion} we arrive at the fully discrete system
\begin{subequations}\label{eqn:FullyDiscrete}
\begin{align}
    \boldsymbol u^{(\rm k+1)}_{\rm m,n} = \boldsymbol u_n - \Delta t \sum_{\ell=0}^{\rm M} \Theta_{\rm m,\ell}\sum_{p=0}^{\rm k} A_{\mathcal{V}}^{\rm p}\mathcal{D}_{\mathcal{V}}^{-1} F^{(\rm k-\rm p)}_{\ell,n}\mathds{1}_{\mathcal E},\label{eqn:DeCmomentum}\\
    e^{(\rm k+1)}_{\rm m,n} = e_n + \Delta t \sum_{\ell=0}^{\rm M} \Theta_{\rm m,\ell}\sum_{p=0}^{\rm k} A^{\rm p}_{\mathcal E}\mathcal{D}_{\mathcal{E}}^{-1} \Big(F^{(\rm k-\rm p)}_{\ell,n}\Big)^T \boldsymbol{u}^{(\rm k -\rm p + 1)}_{\ell,n},\label{eqn:DeCinternal_energy}\\
    \boldsymbol x^{(\rm k+1)}_{\rm m,n} = \boldsymbol x_n + \Delta t \sum_{\ell=0}^{\rm M} \Theta_{\rm m,\ell}\boldsymbol u^{(\rm k+1)}_{\ell,n}\label{eqn:DeCmotion},
\end{align}
\end{subequations}
noting the equation of motion uses a standard ODE DeC formulation.  The application of the DeC iteration to the Lagrangian system is 
detailed algorithmically in Algorithm \ref{alg:DeC} below. We see that with the formulation given in \eqref{eqn:NeumannDeC}, the precomputed 
lumped mass matrices and ${A}_{\mathcal{V}}$, ${A}_{\mathcal{E}}$ matrices can be applied recursively to the current correction residuals.
We note that a convergence criterion can be added to the algorithm as suggested in \cite{Speck.2015} to adaptively select the number of 
correction steps to achieve a desired error.  We note that the storage requirements for the DeC iteration are similar to those of an explicit
Runge-Kutta integrator since only the values at the sub-time nodes are stored, equivalent to the stages of a Runge-Kutta integrator, and we
need not explicitly store the correction levels. 
\begin{algorithm}
\caption{DeC Iteration}
\label{alg:DeC}
\begin{algorithmic}[1]
\State \textbf{Input:} Initial conditions $\boldsymbol{x}_n, \boldsymbol{u}_n, e_n$ for time interval $I_n$ and specify $\rm M$ and $\rm K$;
\State \textbf{Precompute:} $\mathcal{D}_{\mathcal V}^{-1},\;\mathcal{D}_{\mathcal E}^{-1},\;
        A_{\mathcal V} = (I - \mathcal D_{\mathcal V}^{-1}\mathcal M_{\mathcal V}),\;
        A_{\mathcal E} = (I - \mathcal D_{\mathcal E}^{-1}\mathcal M_{\mathcal E})$
\State \textbf{Initialization:}
\For{$\rm m = 1,\ldots,\rm M$ \textbf{and} $\rm k=0$}
    \State $\boldsymbol{x}^{(0)}_{\rm m,n} \gets \boldsymbol{x}_n,\;
            \boldsymbol{u}^{(0)}_{\rm m,n}\gets \boldsymbol{u}_n,\;
            e^{(0)}_{\rm m,n} \gets e_n$
\EndFor
\For{$\rm k = 1,\ldots,\rm K$ \textbf{and} $\rm m=0$}
    \State $\boldsymbol{x}^{(\rm k)}_{\rm 0,n} \gets \boldsymbol{x}_n,\;
            \boldsymbol{u}^{(\rm k)}_{\rm 0,n}\gets \boldsymbol{u}_n,\;
            e^{(\rm k)}_{\rm 0,n} \gets e_n$
\EndFor
\For{$\ell = 0,\ldots,\rm M$}\Comment{Seed the Neumann accumulators.}
    \State $\boldsymbol\delta_{\ell} \gets \boldsymbol 0,\;
            \eta_{\ell} \gets 0$
\EndFor
\State \textbf{Iteration:}
\For{$\rm k = 0,\ldots,\rm K-1$}
    \For{$\ell = 0,\ldots,\rm M$}\Comment{Assemble force tensor and kinematic accumulator.}
        \State Assemble $F^{(\rm k)}_{\ell,n}$
        \State $\boldsymbol\delta_{\ell} \gets \mathcal D_{\mathcal V}^{-1}
                F^{(\rm k)}_{\ell,n}\mathds{1}_{\mathcal E}
                + A_{\mathcal V}\,\boldsymbol\delta_{\ell}$
    \EndFor
    \For{$\rm m = 1,\ldots,\rm M$}\Comment{Momentum update.}
        \State $\boldsymbol u^{(\rm k+1)}_{\rm m,n} \gets \boldsymbol u_n
                - \Delta t \sum_{\ell=0}^{\rm M} \Theta_{\rm m,\ell}\,
                  \boldsymbol\delta_{\ell}$
    \EndFor
    \For{$\ell = 0,\ldots,\rm M$}\Comment{Thermodynamic accumulator.}
        \State $\eta_{\ell} \gets \mathcal D_{\mathcal E}^{-1}
                \big(F^{(\rm k)}_{\ell,n}\big)^T\boldsymbol{u}^{(\rm k+1)}_{\ell,n}
                + A_{\mathcal E}\,\eta_{\ell}$
    \EndFor
    \For{$\rm m = 1,\ldots,\rm M$}\Comment{Internal energy update.}
        \State $e^{(\rm k+1)}_{\rm m,n} \gets e_n
                + \Delta t \sum_{\ell=0}^{\rm M} \Theta_{\rm m,\ell}\,\eta_{\ell}$
    \EndFor
    \For{$\rm m = 1,\ldots,\rm M$}\Comment{Mesh motion.}
        \State $\boldsymbol x^{(\rm k+1)}_{\rm m,n} \gets \boldsymbol x_n
                + \Delta t \sum_{\ell=0}^{\rm M} \Theta_{\rm m,\ell}\,
                  \boldsymbol u^{(\rm k+1)}_{\ell,n}$
    \EndFor
    \If{$\rm k < \rm K-1$}
        \For{$\ell = 0,\ldots,\rm M$}
            \Comment{Update state-dependent quantities.}
            \State $J^{(\rm k+1)}_{\ell} \gets \boldsymbol x^{(\rm k+1)}_{\ell,n}$
                   \Comment{Update the Jacobian data from the new positions.}
            \State $\rho^{(\rm k+1)}_{\ell} \gets
                    \rho_0\,\dfrac{\det J_0}{\det J^{(\rm k+1)}_{\ell}}$
            \State $\boldsymbol\sigma^{(\rm k+1)}_{\ell},\,p^{(\rm k+1)}_{\ell},\,
                    c_{s,\ell}^{(\rm k+1)}\gets
                    \mathrm{EoS}\big(\rho^{(\rm k+1)}_{\ell}, e^{(\rm k+1)}_{\ell,n}\big)$
        \EndFor
    \EndIf
\EndFor
\State \textbf{Finalize:}
\State $\boldsymbol{u}_{n+1} \gets \boldsymbol{u}^{(\rm K)}_{\rm M,n},\;
        e_{n+1} \gets e^{(\rm K)}_{\rm M,n},\;
        \boldsymbol{x}_{n+1} \gets \boldsymbol{x}^{(\rm K)}_{\rm M,n},\;
        \rho_{n+1} \gets \rho_0\,\dfrac{\det J_0}{\det J^{(\rm K)}_{\rm M}}$
\State $\boldsymbol\sigma_{n+1},\,p_{n+1},\,c_{s,n+1} \gets
        \mathrm{EoS}(\rho_{n+1}, e_{n+1})$
\end{algorithmic}
\end{algorithm}

\subsection{Conservation}\label{sec:DeCConservation}
We now investigate the conservation properties of the cG-DeC method applied to the Lagrangian hydrodynamics equations.  To begin, we recall the following properties.
\begin{proposition}
  Let $A = (I-\mathcal{D}^{-1}\mathcal M)$ where $\mathcal D$ is the diagonal matrix formed 
  from summing the rows of the SPD matrix $\mathcal M$. Then, we have the following:
  \begin{enumerate}[label=\roman*]
    \item $A \mathcal D^{-1} = \mathcal{D}^{-1}A^T$,\label{prop:one}
    \item $\mathcal D A \mathcal D^{-1} = A^T$,\label{prop:two}
    \item $\mathcal{D} A^{\rm p} \mathcal{D}^{-1} = \Big(A^T\Big)^{\rm p},$ for $\rm p\geq 0$, \label{prop:three}
    \item $\mathds{1}A^T = \boldsymbol 0^T$.\label{prop:four}
  \end{enumerate}
\end{proposition}
\begin{proof}
Properties \eqref{prop:one} and \eqref{prop:two} follow from direct computation. To prove \eqref{prop:three}, write $\mathcal DA^{\rm p} \mathcal D^{-1} = \mathcal DA^{\rm p-1} A\mathcal D^{-1} = \mathcal D A^{\rm p-1} \mathcal D^{-1}A^T$ which follows from property \eqref{prop:one} above.  Then, by repeated application
we have $$\mathcal D A^{\rm p-1} \mathcal D^{-1}A^T = \mathcal D A^{\rm p -2} \mathcal D^{-1} \Big( A^T \Big)^2 = \ldots = \mathcal{D}A\mathcal{D}^{-1}\Big(A^T\Big)^{\rm p -1}$$ and the result follows from applying property \eqref{prop:two} to obtain the final power. Property \eqref{prop:four} follows from a direct computation and the fact that $\mathcal M$ is symmetric and $\mathcal D$ is its row sum.
\end{proof}

We now arrive at our first result. 

\begin{proposition}
    The momentum is conserved at the fully discrete level. 
\end{proposition}
\begin{proof}
To see this, consider the following.
  Using the row-sum lumped mass and symmetry of $\mathcal{M}_{\mathcal{V}}$, we have $\mathds{1}_{\mathcal{V}}\mathcal M_{\mathcal{V}}=\mathds{1}_{\mathcal{V}}\mathcal D_{\mathcal{V}}$, so that
  $$
   \mathds{1}_{\mathcal{V}}\mathcal D_{\mathcal{V}}\boldsymbol{u}_{n+1} = \mathds{1}_{\mathcal{V}}\mathcal D_{\mathcal{V}}\boldsymbol{u}_{n} - \Delta t \sum_{\ell=0}^{\rm M}\Theta_{\rm M, \ell} \sum_{\rm p =0}^{\rm K-1} \mathds{1}_{\mathcal{V}} \mathcal D_{\mathcal{V}}A_{\mathcal{V}}^{\rm p} \mathcal D_{\mathcal{V}}^{-1} F^{(\rm K -\rm p-1)}_{\ell, n}\mathds{1}_{\mathcal{E}}.
  $$
  We have by property \eqref{prop:three} that
  \begin{equation*}
    \begin{aligned}
      \mathds{1}_{\mathcal{V}}\mathcal D_{\mathcal{V}}\boldsymbol{u}_{n+1} = \mathds{1}_{\mathcal{V}}\mathcal D_{\mathcal{V}}\boldsymbol{u}_{n} - \Delta t \sum_{\ell=0}^{\rm M}\Theta_{\rm M, \ell} \sum_{\rm p =0}^{\rm K-1} \mathds{1}_{\mathcal{V}} \Big(A_{\mathcal{V}}^T\Big)^{\rm p}F^{(\rm K -\rm p-1)}_{\ell, n}\mathds{1}_{\mathcal{E}}\\
      = \mathds{1}_{\mathcal{V}}\mathcal D_{\mathcal{V}}\boldsymbol{u}_{n} - \Delta t \sum_{\ell=0}^{\rm M}\Theta_{\rm M, \ell} \mathds{1}_{\mathcal{V}} F^{(\rm K-1)}_{\ell, n}\mathds{1}_{\mathcal{E}}\\
      = \mathds{1}_{\mathcal{V}}\mathcal D_{\mathcal{V}}\boldsymbol{u}_{n},
    \end{aligned}     
  \end{equation*}
  where we used property \eqref{prop:four} for $\rm p \geq 1$ and $\mathds{1}_{\mathcal{V}}F^{(\rm K-1)}_{\ell,n}=\boldsymbol{0}$ by the partition of unity property for the kinematic basis for all $\ell\in\{0,\ldots,\rm M\}$ and independent of the correction level.
\end{proof}

\subsubsection{Total Energy Conservation}
To conserve the total energy, we must modify the internal energy equation update. This is due to the corrections in the DeC method needed to account for the discrepancy between the low-order and 
high-order operator.  First, let us consider the change in the internal energy.  We denote this quantity $\delta IE$ and it follows naturally from \eqref{eqn:DeCinternal_energy} that it is defined to be
\begin{equation}
  \begin{aligned}
    \delta IE = \mathds{1}_{\mathcal{E}}\mathcal{D}_{\mathcal{E}} (e_{n+1}-e_{n}) = \Delta t\sum_{\ell=0}^{\rm M}\Theta_{\rm M, \ell} \sum^{\rm K-1}_{\rm p = 0}\mathds{1}_{\mathcal{E}}\mathcal{D}_{\mathcal{E}}A^{\rm p}_{\mathcal{E}}\mathcal D_{\mathcal{E}}^{-1} \Big( F^{(\rm K - \rm p -1)}_{\ell, n} \Big)^T\boldsymbol u^{(\rm K -\rm p)}_{\ell,n}\\
              = \Delta t\sum_{\ell=0}^{\rm M}\Theta_{\rm M, \ell} \sum^{\rm K-1}_{\rm p = 0}\mathds{1}_{\mathcal{E}}\Big(A^T_{\mathcal{E}}\Big)^{\rm p} \Big( F^{(\rm K - \rm p-1)}_{\ell, n} \Big)^T\boldsymbol u^{(\rm K -\rm p)}_{\ell,n}\\
              = \Delta t\sum_{\ell=0}^{\rm M}\Theta_{\rm M, \ell}\mathds{1}_{\mathcal{E}}\Big( F^{(\rm K-1)}_{\ell, n} \Big)^T\boldsymbol u^{(\rm K )}_{\ell,n},\\
  \end{aligned}
\end{equation} 
using properties \eqref{prop:three} and \eqref{prop:four} again.  Using the update for the momentum, we get
\begin{equation}\label{eqn:deltaIEnoexpansioninU}
  \begin{aligned}
    \delta IE = \Delta t\sum_{\ell=0}^{\rm M}\Theta_{\rm M, \ell}\mathds{1}_{\mathcal{E}}\Big( F^{(\rm K-1)}_{\ell, n} \Big)^T\Big(\boldsymbol u_n - \Delta t \sum_{\rm s=0}^{\rm M}\Theta_{\ell,\rm s}\sum_{\rm r=0}^{\rm K-1} A_{\mathcal{V}}^{\rm r}\mathcal{D}^{-1}_{\mathcal{V}}F^{(\rm K - \rm r-1)}_{\rm s,n}\mathds{1}_{\mathcal{E}}\Big),\\
              = \Delta t\sum_{\ell=0}^{\rm M}\Theta_{\rm M, \ell}\mathds{1}_{\mathcal{E}}\Big( F^{(\rm K-1)}_{\ell, n} \Big)^T\boldsymbol u_n - \Delta t^2 \sum_{\ell,\rm s=0}^{\rm M} \Theta_{\rm M, \ell}\Theta_{\ell, \rm s}\sum_{\rm r=0}^{\rm K-1} \mathds{1}_{\mathcal{E}}\Big(F^{(\rm K-1)}_{\ell, n}\Big)^T A_{\mathcal{V}}^{\rm r}\mathcal{D}_{\mathcal{V}}^{-1}F^{(\rm K-r-1)}_{s,n}\mathds{1}_{\mathcal{E}}.
  \end{aligned}
\end{equation}

Similarly, the change in kinetic energy, denoted $\delta KE$, is
\begin{equation}
  \begin{aligned}
  \delta KE = \frac12 \boldsymbol u_{n+1}\mathcal{D}_{\mathcal{V}}\boldsymbol{u}_{n+1} - \frac12\boldsymbol u_n\mathcal{D}_{\mathcal{V}}\boldsymbol u_n  = -\Delta t \sum_{\ell=0}^{\rm M} \Theta_{\rm M, \ell} \sum_{\rm p=0}^{\rm K-1}\mathds{1}_{\mathcal{E}}\Big(F^{(\rm K-\rm p-1)}_{\ell,n}\Big)^TA_{\mathcal{V}}^{\rm p}\boldsymbol{u}_n \\
              + \frac12 \Delta t^2 \sum_{\ell, \rm s =0}^{\rm M}\Theta_{\rm M, \ell}\Theta_{\rm M, \rm s} \sum_{\rm p, \rm r=0}^{\rm K-1} \mathds{1}_{\mathcal{E}} \Big(F^{(\rm K - \rm p-1)}_{\ell, n}\Big)^T A_{\mathcal{V}}^{\rm p + \rm r}\mathcal{D}_{\mathcal{V}}^{-1}F^{(\rm K - \rm r-1)}_{\rm s, n}\mathds{1}_{\mathcal{E}},
  \end{aligned}
\end{equation}
by using the update defined in equation \eqref{eqn:DeCmomentum} and expanding.  Now, considering the change in total energy, $\delta TE$, we have
\begin{equation}
  \begin{aligned}
    \delta TE = \delta KE + \delta IE = -\Delta t \sum_{\ell=0}^{\rm M} \Theta_{\rm M, \ell} \sum_{\rm p=1}^{\rm K-1}\mathds{1}_{\mathcal{E}}\Big(F^{(\rm K-\rm p-1)}_{\ell,n}\Big)^T A_{\mathcal{V}}^{\rm p} \boldsymbol{u}_n \\+ \frac12 \Delta t^2 \sum_{\ell, \rm s =0}^{\rm M}\Theta_{\rm M, \ell}\Theta_{\rm M, \rm s} \sum_{\rm p, \rm r=0}^{\rm K-1} \mathds{1}_{\mathcal{E}} \Big(F^{(\rm K - \rm p-1)}_{\ell, n}\Big)^TA_{\mathcal{V}}^{\rm p + \rm r}\mathcal{D}_{\mathcal{V}}^{-1}F^{(\rm K - r-1)}_{\rm s, n}\mathds{1}_{\mathcal{E}}\\
                          - \Delta t^2 \sum_{\ell,\rm s=0}^{\rm M} \Theta_{\rm M, \ell}\Theta_{\ell, \rm s}\sum_{\rm r=0}^{\rm K-1} \mathds{1}_{\mathcal{E}}\Big(F^{(\rm K-1)}_{\ell, n}\Big)^T A_{\mathcal{V}}^{\rm r}\mathcal{D}_{\mathcal{V}}^{-1}F^{(\rm K-r-1)}_{\rm s,n}\mathds{1}_{\mathcal{E}},
  \end{aligned}
\end{equation}
where we have a cancellation of order $\Delta t$ terms for $\rm p=0$.  The cancellation of the $\Delta t^2$ terms is more complicated. It does not appear to be possible to avoid correcting the internal energy to account for the higher iteration terms with $\rm p,\rm r > 0$ and even for the case $\rm p, \rm r = 0$ the situation is non-trivial.  Let us consider the conditions needed to obtain a cancellation of the $\Delta t^2$ terms for the case $\rm p,\, \rm r = 0$, 
\begin{equation}\label{eqn:dt_sqrd_TE}
    \begin{aligned}
        \frac12 \Delta t^2 \sum_{\ell, \rm s =0}^{\rm M}\Theta_{\rm M, \ell}\Theta_{\rm M, \rm s} \mathds{1}_{\mathcal{E}} \Big(F^{(\rm K-1)}_{\ell, n}\Big)^T\mathcal{D}_{\mathcal{V}}^{-1}F^{(\rm K-1)}_{\rm s, n}\mathds{1}_{\mathcal{E}}
                          \\- \Delta t^2 \sum_{\ell,\rm s=0}^{\rm M} \Theta_{\rm M, \ell}\Theta_{\ell, \rm s}\mathds{1}_{\mathcal{E}}\Big(F^{(\rm K-1)}_{\ell, n}\Big)^T \mathcal{D}_{\mathcal{V}}^{-1}F^{(\rm K-1)}_{\rm s,n}\mathds{1}_{\mathcal{E}}=0.
    \end{aligned}
\end{equation}
Immediately, one can see that a cancellation is not achievable with the DeC integrators without modifying the velocity contribution to the work term. Such an approach was taken in \cite{Sandu.2021} in the case of Runge-Kutta integrators, where a formally implicit-explicit partitioning was employed for the system \eqref{eqn:Euler} which made use of an auxiliary velocity variable that employs an implicit, symplectic tableau.  Such an extension may be possible for arbitrary order DeC integrators, but this is far beyond the scope of the present work. 

However, in the second-order case, the modification is simple enough and we provide it here.  To this end, define the time averaged velocity at correction $\rm k$ and time interval $I_n$ for $\rm M = 1$ to be $$\bar{\bs u}_n^{(\rm k)}=\frac{\bs u^{(\rm k)}_{1,n} + \bs u^{(\rm k)}_{0,n}}{2}.$$  Then the modified fully-discrete, second-order update is 
\begin{subequations}\label{eqn:2ndOrderFullyDiscreteModified}
\begin{align}
    \boldsymbol u^{(\rm k+1)}_{\rm 1,n} = \boldsymbol u_n - \frac{\Delta t}{2} \sum_{p=0}^{\rm k} A_{\mathcal{V}}^{\rm p}\mathcal{D}_{\mathcal{V}}^{-1} \Big(F^{(\rm k-\rm p)}_{1,n}+F^{(\rm k-\rm p)}_{0,n}\Big)\mathds{1}_{\mathcal E},\\
    e^{(\rm k+1)}_{\rm 1,n} = e_n + \frac{\Delta t}{2} \sum_{p=0}^{\rm k} A^{\rm p}_{\mathcal E}\mathcal{D}_{\mathcal{E}}^{-1} \Big(F^{(\rm k-\rm p)}_{1,n}+F^{(\rm k-\rm p)}_{0,n}\Big)^T \bar{\boldsymbol{u}}^{(\rm k + 1)}_{n},\\
    \boldsymbol x^{(\rm k+1)}_{\rm 1,n} = \boldsymbol x_n + \Delta t \bar{\boldsymbol{u}}^{(\rm k + 1)}_{n}.
\end{align}
\end{subequations}
Now consider the total energy change using \eqref{eqn:2ndOrderFullyDiscreteModified}. We have 
\begin{equation}
    \begin{aligned}
        \delta TE = \frac{\Delta t}{2}\mathds{1}_{\mathcal{E}}\Big( F^{(K-1)}_{1,n} + F^{(K-1)}_{0,n} \Big)^T \bar{\bs u}^{(\rm K)}_{n} - \frac{\Delta t}{2} \bar{\bs u}^{(\rm K)}_{n} \sum_{\rm p = 0}^{\rm K - 1} \Big(A_{\mathcal{V}}^T\Big)^{\rm p} \Big(F^{(\rm K - \rm p -1)}_{1,n} + F^{(\rm K - \rm p -1)}_{0,n}\Big)\mathds{1}_{\mathcal{E}}\\
        = -\frac{\Delta t}{2} \bar{\bs u}^{(\rm K)}_{n}\sum_{\rm p = 1}^{\rm K - 1} \Big(A_{\mathcal{V}}^T\Big)^{\rm p} \Big(F^{(\rm K - \rm p -1)}_{1,n} + F^{(\rm K - \rm p -1)}_{0,n}\Big)\mathds{1}_{\mathcal{E}},
    \end{aligned}
\end{equation}
where we have used $\delta KE = \bar{\bs u}^{(\rm K)}_n \mathcal{D}_{\mathcal{V}} (\bs u_{n+1} - \bs u_n)$ for the definition of the kinetic energy which is algebraically equivalent to the definition used in the previous discussion.

Therefore, we can provide a correction for the modified second-order DeC method \eqref{eqn:2ndOrderFullyDiscreteModified} which conserves the total energy and is independent of the number of iterations.
\begin{proposition}
  For the modified second-order system \eqref{eqn:2ndOrderFullyDiscreteModified} with $\rm K$ corrections, applying the total energy correction term
  \begin{equation}\label{eqn:deltaF}
      \delta F = \frac{\Delta t}{2}\sum_{\rm p=1}^{\rm K-1}\mathcal{D}_{\mathcal{E}}^{-1}\Big(F^{(\rm K-\rm p-1)}_{1,n} + F^{(\rm K-\rm p-1)}_{0,n} \Big)^TA_{\mathcal{V}}^{\rm p} \bar{\boldsymbol{u}}^{(\rm K)}_{n},
  \end{equation}
   to the internal energy update in \eqref{eqn:2ndOrderFullyDiscreteModified} at the final iteration conserves the total energy. 

\end{proposition}
\begin{proof}
  Applying the proposed correction to the final iteration of the internal energy equation in \eqref{eqn:2ndOrderFullyDiscreteModified} above and recomputing $\delta IE$ as in \eqref{eqn:deltaIEnoexpansioninU} above, we get
  \begin{equation}
    \begin{aligned}
      \delta IE = \frac{\Delta t}{2}\mathds{1}_{\mathcal{E}}\Big( F^{(\rm K-1)}_{1, n}+F^{(\rm K-1)}_{0, n} \Big)^T\bar{\boldsymbol u}^{(\rm K)}_n + \frac{\Delta t}{2} \sum_{\rm p=1}^{\rm K-1}\mathds{1}_{\mathcal{E}}\Big(F^{(\rm K-\rm p-1)}_{1,n}+ F^{(\rm K-\rm p-1)}_{0,n}\Big)^T A_{\mathcal{V}}^{\rm p} \bar{\boldsymbol{u}}^{(\rm K)}_n.
    \end{aligned}
  \end{equation}
  Adding $\delta KE$ gives the result.
\end{proof}

We note that for the modified second-order discretization \eqref{eqn:2ndOrderFullyDiscreteModified}, if $\mathcal{M}_{\mathcal{V}} = \mathcal{D}_{\mathcal{V}}$ the total energy correction is not required and \eqref{eqn:2ndOrderFullyDiscreteModified} conserves the total energy by construction.  However, in the high-order case with $\rm M > 1$, even if one assumes $\mathcal{M}_{\mathcal{V}} = \mathcal{D}_{\mathcal{V}}$, a further modification to the fully discrete system is required in light of \eqref{eqn:dt_sqrd_TE} above.  On tensor-product grids, one may use high-order mass lumping via collocated quadrature to obtain $\mathcal{M}_{\mathcal{V}}=\mathcal{D}_{\mathcal{V}}$ and eliminate the discrepancy between the lumped and consistent mass matrix. However, from our experiments, under-integration of the force tensor for $d\geq 2$ impacts stability of the method.  Further, the issues just discussed highlighted by \eqref{eqn:dt_sqrd_TE} must still be addressed. In such a setting, that is, collocated mass lumping on tensor-product grids where the need for a Krylov solver is obviated, it is not easy to see the advantage of developing total energy conserving DeC integrators for the system \eqref{eqn:Euler} when Sandu et al. have already solved the problem for Runge-Kutta integrators in \cite{Sandu.2021}.

In the next section we provide numerical demonstrations of the cG-DeC method up to fourth order.  For the second-order case, we compare the cG-DeC method with and without the total energy correction on classic shock test problems.  The second order methods are compared with a continuous Galerkin discretization using a conservative second-order Runge-Kutta integrator \cite{Dobrev.2012}, while our third and fourth order schemes are compared with continuous Galerkin methods using third order strong-stability preserving and fourth order Runge-Kutta integrators, respectively.

\section{Numerical Results}\label{sec:Tests}
Our numerical tests are carried out in the Laghos miniapp \cite{laghos_github} based on the MFEM finite element library \cite{mfem}. As a result, we can provide a direct comparison with the continuous finite element discretization using explicit Runge-Kutta integrators described in \cite{Dobrev.2012}, which we will refer to as the cG-RK method.  The second order cG-DeC method with total energy correction \eqref{eqn:deltaF} will be referred to as the cG-DeC-$\delta F$ method.  The implementation makes use of partial assembly \cite{Bello-Maldonado.2020} when using tensor-product meshes.  First, in this section, we investigate the convergence properties of the cG-DeC method on a smooth test using the limited artificial viscosity described in Section \ref{sec:FE}, considering the number of corrections needed to obtain the expected convergence rates as the mesh is refined.  Concerning the hyperviscosity limiter, a single iteration of the Neumann series approximation to the Laplacian was sufficient for all tests.  
Performance comparisons of the cG-DeC method with the cG-RK method are provided for this smooth test problem.  Next we compare the cG-DeC, cG-DeC-$\delta F$ and cG-RK methods in terms of global $L_1$ error and total energy conservation on the Sedov blast.  Further comparisons are provided on the Noh and three-dimensional triple point problems.

\begin{figure}[ht!]
    \centering
    \includegraphics[width=0.35\linewidth]{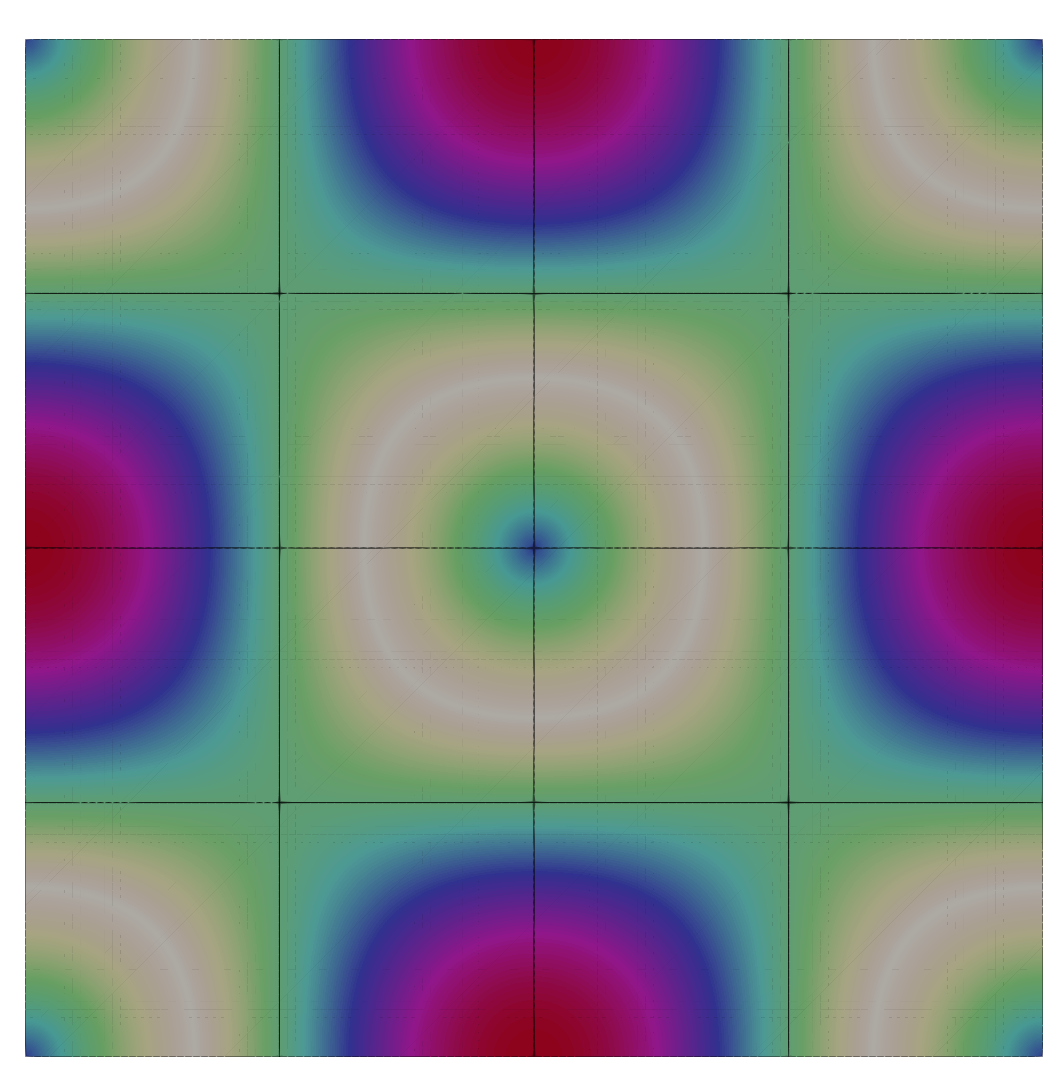}
    \includegraphics[width=0.35\linewidth]{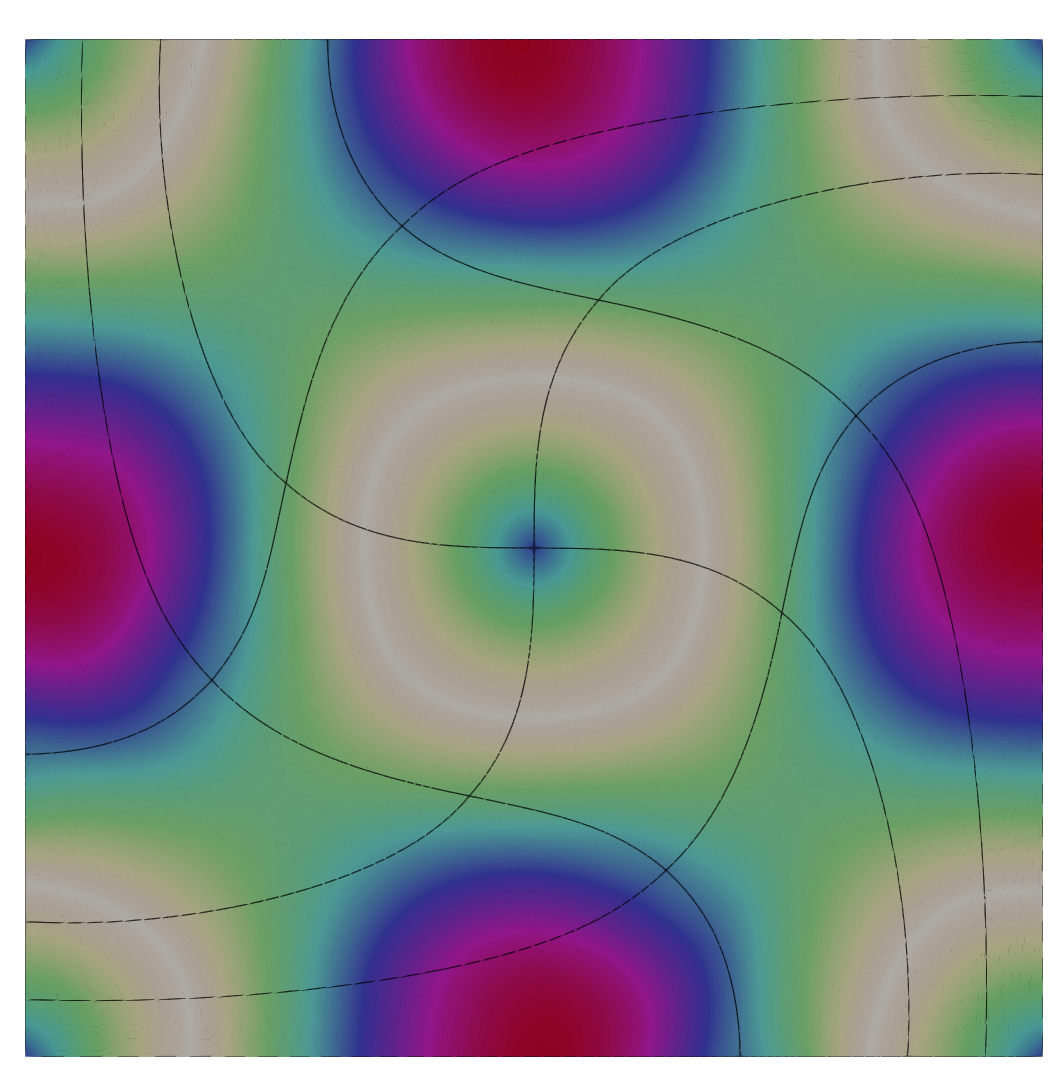}\\
    \includegraphics[width=0.45\linewidth]{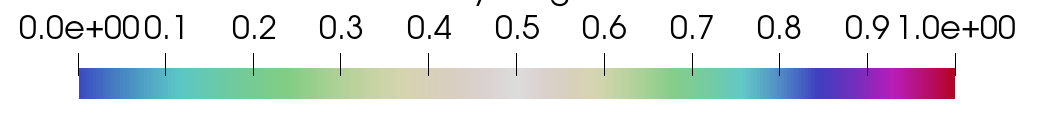}
    \caption{The Taylor-Green vortex problem using $Q_4Q_3$ elements. (Left) The magnitude of the velocity and mesh at $t=0$. (Right) The magnitude of the velocity and mesh at $t=0.5$. }
    \label{fig:TGV_fields}
\end{figure}
\subsection{Taylor-Green Vortex}\label{sec:TGV}
To test the convergence properties of the method we employ the Taylor-Green vortex test problem.  
The problem models a smooth, stationary incompressible flow so that $\bs u(t,\bs x(t,\hat{\bs x})) = \bs u_0(\bs x(0,\hat{\bs x}))$ 
and similarly for the pressure and specific internal energy.  The density is constant so that $\rho \equiv 1$
and $\nabla \cdot \bs u = 0$.  The initial conditions are given by
    \begin{equation}\label{eqn:TG_init_vel}
        \begin{aligned}
            \bs u_0(\bs x) = \begin{pmatrix}
                                \sin(\pi x)\cos(\pi y) \\
                                -\cos(\pi x)\sin(\pi y) 
                             \end{pmatrix},
        \end{aligned}
    \end{equation} 
for the velocity field, using the common notation $\bs x(t,\hat{\bs x}) = (x, y)^T$ above.  
The pressure and specific internal energy are initially set to be
    \begin{equation}\label{eqn:TG_init_pressure_and_sie}
        \begin{aligned}
            p_0(0,\bs x(0,\hat{\bs x})) = \frac14 (\cos(2\pi x) + \cos(2\pi y)) + 1\\
            e_0(0,\bs x(0,\hat{\bs x})) = p_0(0,\bs{x}(0,\hat{\bs x}))/((\gamma - 1)\rho),
        \end{aligned}
    \end{equation} 
where $\gamma = \frac53$.

To maintain stationarity of the fields (see Figure \ref{fig:TGV_fields}) a source term is added using the evolution equation for the specific internal energy and the initial conditions for the velocity and pressure fields \cite{RD_SGH,Dobrev.2012,Guermond.2016.EntropyViscosity}.  The source term is then 
    \begin{equation}\label{eqn:TG_source_term}
        S(\bs x(t,\hat{\bs x})) = \frac{3\pi}{8}(\cos(3\pi x) \cos(\pi y) - \cos(\pi x)\cos(3\pi y)).
    \end{equation}
\begin{figure}[hb!]
    \centering
    \begin{subfigure}{0.36\linewidth}
        \centering
        \raisebox{1.7in}{\includegraphics[width=\linewidth]{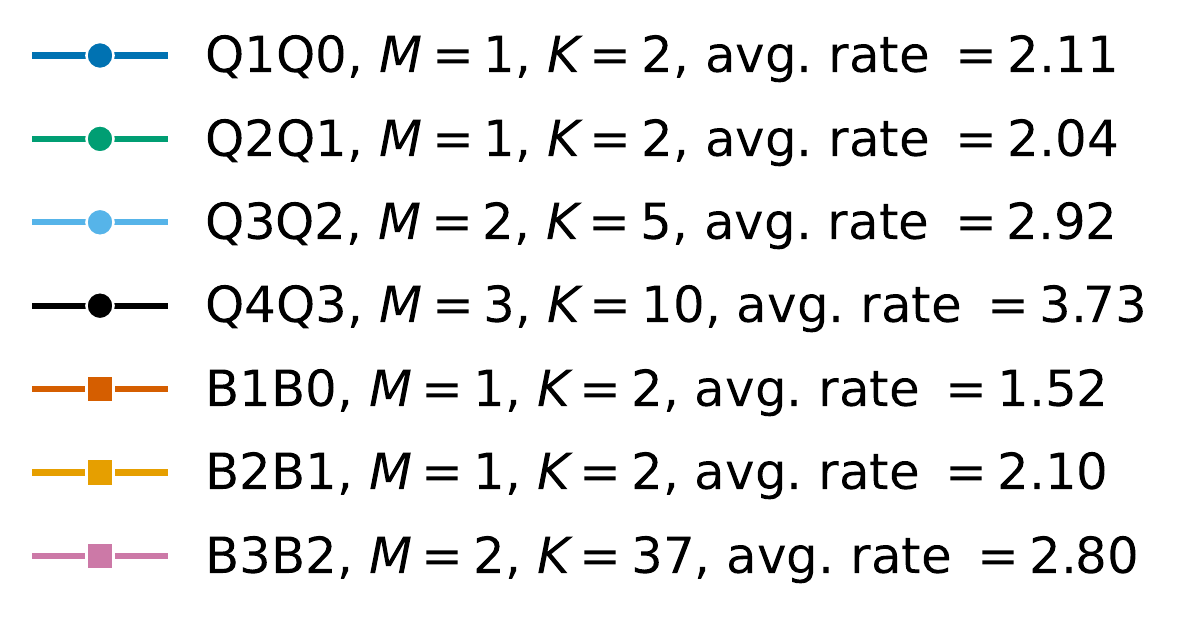}}
        \label{fig:convergence_B3B2}
    \end{subfigure}
    \begin{subfigure}{0.63\linewidth}
        \centering
        \includegraphics[width=\linewidth]{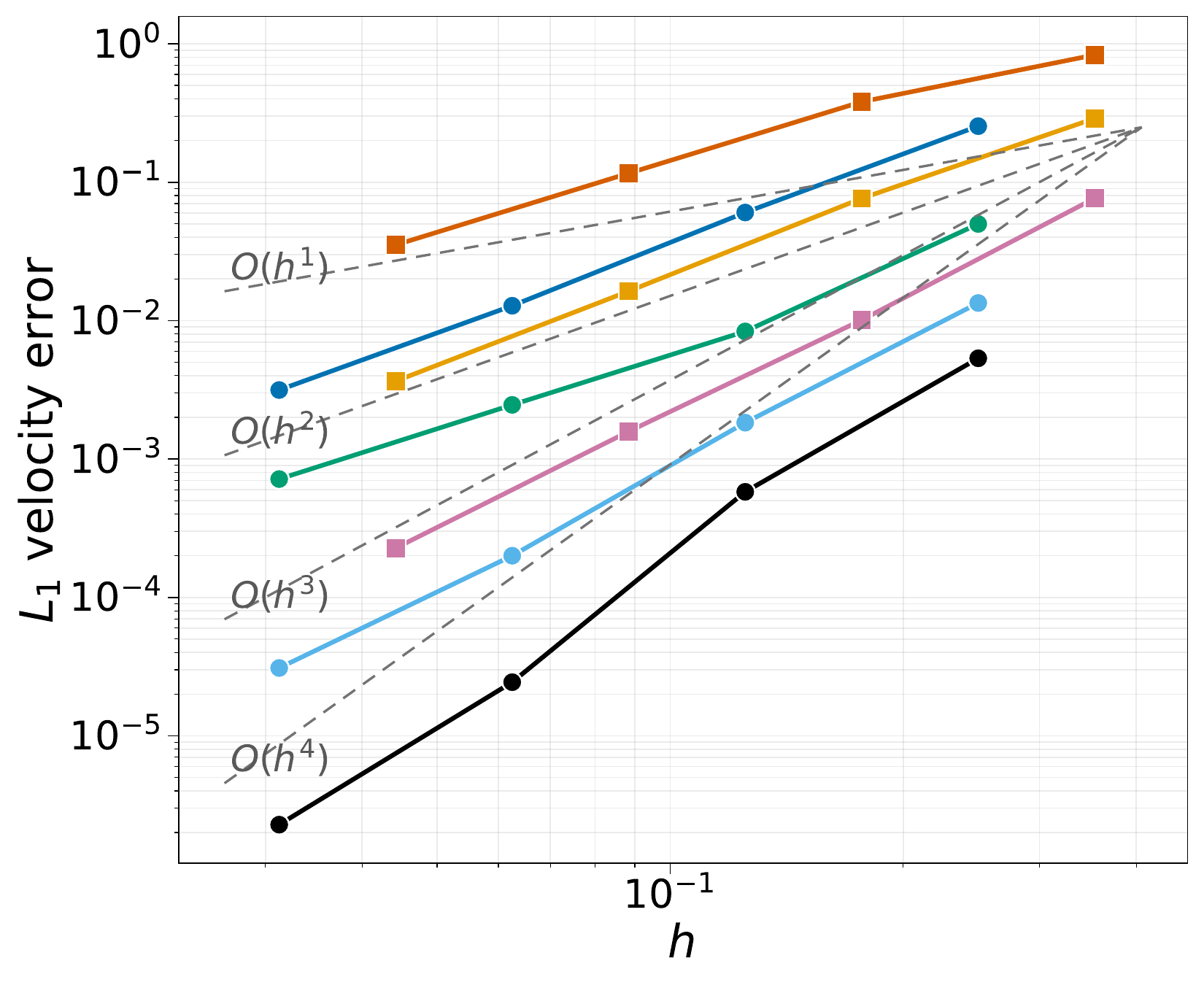}
        \label{fig:convergence_B2B1}
    \end{subfigure}
    \caption{Convergence rates for cG-DeC method on quadrilateral and triangular meshes for the Taylor-Green vortex problem. Lagrange polynomials are used on the quad meshes, indicated by $Q_{\rm q}Q_{\rm q-1}$ whereas triangles require Bernstein polynomials to avoid lumped mass degeneracy at all orders and are denoted $B_{\rm q}B_{\rm q-1}$ in the legend. For $\rm q=1,2$ we use the second-order DeC time integrator. For $\rm q > 2$, we use the corresponding $\rm q$th-order DeC integrator. The number of corrections needed to achieve convergence at all mesh resolutions and the average convergence rates for each method are reported in the legend.}
    \label{fig:convergence_tgv}
\end{figure}
In our convergence tests, we keep the $\rm{CFL}=0.5$ fixed.  In Figure \ref{fig:convergence_tgv}, we plot the convergence rates as measured in the $L_1$ norm against the exact solution.  For each polynomial order, we run the convergence study over a range of $\rm K$ values, which indicated that at the finer mesh scales more iterations are needed to obtain the expected convergence rates.  The convergence tests are run to $t=0.5$.  The tests begin with the number of corrections equal to the kinematic polynomial order and increase until the expected convergence is reached at all mesh refinement levels. The computed convergence rates and the number of corrections needed to obtain those convergence rates are given in the legend of Figure \ref{fig:convergence_tgv}. In each case, we leave the viscosity turned on, using the limiter described in Section \ref{sec:FE} above. We found that a single iteration of the Neumann series approximation to the Laplacian was enough to obtain the expected convergence rates across all tests.  The cG-DeC method is tested on both quadrilateral meshes with Lagrange basis functions using the GLL DoF distribution and on triangles with Bernstein basis functions.

For linear and quadratic elements, using both the Lagrange basis and the Bernstein basis, the number of corrections is optimal with the second-order DeC integrator {\emph{for these mesh resolutions}}. We see that for the $Q_3Q_2$ and $Q_4Q_3$ elements, we need five and ten corrections with the third- and fourth-order DeC integrator, respectively, to achieve the expected order of convergence.  For the $B_3B_2$ elements, we need 37 corrections to obtain the expected 3rd order convergence rates.  This is in agreement with the fact that the iteration depends on the powers of $A_{\mathcal{V}}$ and for Bernstein elements the spectral radius of $A_{\mathcal{V}}$ is very close to unity.
\begin{table}[hb!]
\centering
\begin{tabular}{llrrrrrrl}
\toprule
Order & Scheme & $n_{\Delta t}$ & $n_F$ & $n_{\mathcal{M}_{\mathcal{V}}}$ & $n_{\mathcal{M}_{\mathcal{E}}}$ & $n_Q$ & Total &
${L_1}$ Error \\
\midrule
$Q_1Q_0$ & cG-RK (RK2Avg) & 7 & 36 & 176 & 18 & 21 & 251 & $4.71\times 10^{-3}$ \\
$Q_1Q_0$ & cG-DeC ($M{=}1,K{=}2$) & 7 & 63 & 36 & 18 & 73 & 190 & $4.71\times 10^{-3}$ \\
\midrule
$Q_2Q_1$ & cG-RK (RK2Avg) & 13 & 60 & 440 & 30 & 33 & 563 & $1.83\times 10^{-4}$ \\
$Q_2Q_1$ & cG-DeC ($M{=}1,K{=}2$) & 13 & 105 & 60 & 30 & 121 & 316 & $1.44\times 10^{-4}$ \\
\midrule
$Q_3Q_2$ & cG-RK (RK3 SSP) & 20 & 132 & 510 & 198 & 69 & 909 & $4.83\times 10^{-6}$ \\
$Q_3Q_2$ & cG-DeC ($M{=}2,K{=}5$) & 20 & 616 & 528 & 264 & 639 & 2047 & $4.81\times 10^{-6}$ \\
\bottomrule
\end{tabular}
\caption{We measure the number of time steps, force tensor evaluations, kinematic and thermodynamic mass matrix applies and function evaluations at the quadrature points for the cG-DeC and cG-RK methods. The $L_1$ error of the velocity is also presented.  We see that for $Q_1Q_0$ and $Q_2Q_1$, the cG-DeC method has better total efficiency due entirely to the expense of the preconditioned conjugate gradient method in the cG-RK algorithm. For $Q_3Q_2$, the situation reverses owing to the increased number of corrections needed to obtain convergence.}
\label{tab:tgv_comparison}
\end{table}

Lastly, we compare the cG-RK method with the cG-DeC method.  In Table \ref{tab:tgv_comparison}, we give values for the number of time steps, $n_{\Delta t}$; the number of force tensor evaluations, $n_F$; the number of kinematic mass matrix applies, $n_{\mathcal{M}_{\mathcal{V}}}$; the number of thermodynamic mass matrix applies, $n_{\mathcal{M}_{\mathcal{E}}}$; the number of function evaluations at the quadrature points, $n_Q$; and the sum of $n_F$, $n_{\mathcal{M}_{\mathcal{V}}}$, $n_{\mathcal{M}_{\mathcal{E}}}$ and $n_Q$ represented as ``Total'' in Table \ref{tab:tgv_comparison}.  Further, we report the $L_1$ error for each method across the  element types.  For the obtained values, we run the TGV test until $t=0.1$ with the same CFL used in the convergence analysis.  The errors for the cG-DeC method are equal to or very slightly less than those obtained with the cG-RK method.  Considering the convergence studies above, it is unsurprising, given the additional number of corrections, that the cG-RK method is more efficient than the cG-DeC method for $Q_3Q_2$ elements.  However, for $Q_1Q_0$ and $Q_2Q_1$ elements, the cG-DeC method shows better efficiency overall.  

\subsection{Sedov}\label{sec:Sedov}
In the Taylor--Sedov (Taylor--von Neumann--Sedov) blast wave 
\cite{Sedov.1993.Similarity, Taylor.1950.BlastWaveI, Taylor.1950.BlastWaveII, Shestakov.1999, Kamm.2000}
an energy source is placed at the origin, which is quickly converted to kinetic
energy and produces a very strong, spherically symmetric blast wave. 
The test provides a means of verifying a Lagrangian scheme's ability to maintain
symmetry and stable mesh deformations while capturing the correct shock location,
an indication that energy conservation is maintained.
\begin{table}[h]
\centering
\small
\setlength{\tabcolsep}{4pt}
\begin{tabular}{lrrrrrrr}
\toprule
& & \multicolumn{3}{c}{$\delta TE$} & \multicolumn{3}{c}{$\|\rho-\rho_{\mathrm{ex}}\|_{L_1}$} \\
\cmidrule(lr){3-5}\cmidrule(lr){6-8}
Order & DoFs
& cG-RK & cG-DeC-$\delta F$ & cG-DeC
& cG-RK & cG-DeC-$\delta F$ & cG-DeC \\
\midrule
$Q_1Q_0$ & 7442 & $4.52{\times}10^{-10}$ & $5.55{\times}10^{-16}$ & $1.80{\times}10^{-3}$ & $1.37{\times}10^{-1}$ & $1.62{\times}10^{-1}$ & $1.21{\times}10^{-1}$ \\
$Q_2Q_1$ & 7442 & $3.37{\times}10^{-13}$ & $2.78{\times}10^{-16}$ & $2.72{\times}10^{-4}$ & $1.51{\times}10^{-1}$ & $1.69{\times}10^{-1}$ & $1.49{\times}10^{-1}$ \\
$Q_3Q_2$ & 7442 & $3.27{\times}10^{-4}$ & {--- } & $1.02{\times}10^{-5}$ & $1.49{\times}10^{-1}$ & {--- } & $1.54{\times}10^{-1}$ \\
$Q_4Q_3$ & 7442 & $1.85{\times}10^{-4}$ & {--- } & $8.98{\times}10^{-6}$ & $1.50{\times}10^{-1}$ & {--- } & $1.53{\times}10^{-1}$ \\
\bottomrule
\end{tabular}
\caption{ Holding the number of kinematic DoFs fixed, we measure the total energy conservation error, $\delta TE$, and $L_1$ error against the semi-analytic density for the cG-RK, cG-DeC-$\delta F$ and cG-DeC methods. For the $Q_1Q_0$ and $Q_2Q_1$, the cG-RK method uses the RK2Avg scheme, which is total energy conserving, and the SSP-RK3 and the RK4 integrators for $Q_3Q_2$ and $Q_4Q_3$ elements, respectively.  Similarly, for $Q_1Q_0$ and $Q_2Q_1$, we use $\rm M =1$ and $\rm K = 2$ for both the cG-DeC and cG-DeC-$\delta F$ methods where the latter has the total energy correction \eqref{eqn:deltaF} applied.  For the $Q_3Q_2$ and $Q_4Q_3$ elements, we use $\rm M = 2$ and $\rm K = 5$ and $\rm M = 3$ and $\rm K = 10$, respectively.}
\label{tab:sedov_table}
\end{table}
\begin{figure}[ht!]
    \centering
    \raisebox{2.5in}{\includegraphics[width=0.15\linewidth]{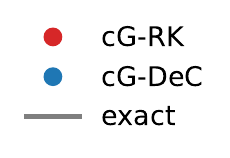}}
    \includegraphics[width=0.84\linewidth]{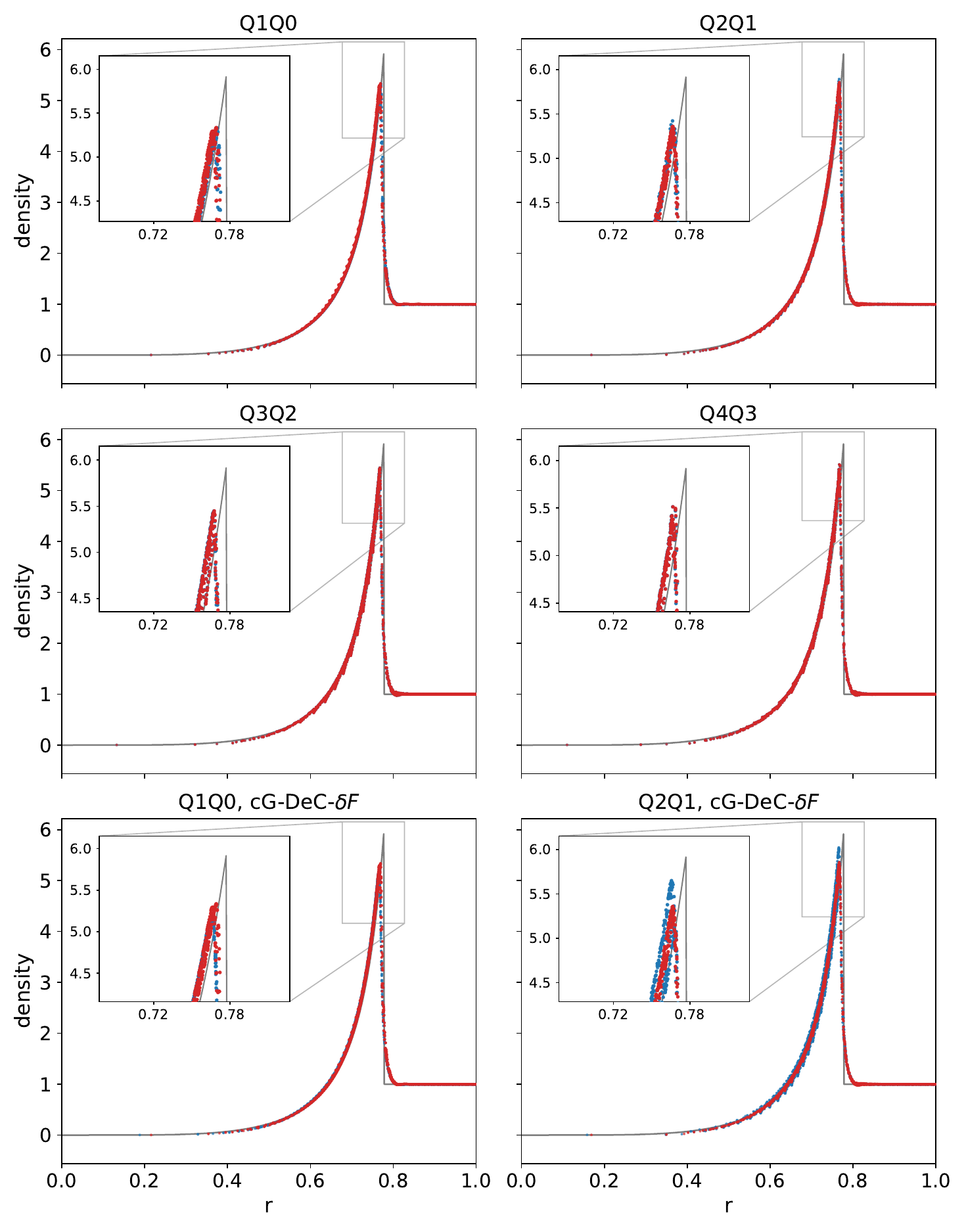}
    \caption{The Sedov blast wave for various finite element discretizations with the number of DoF kept fixed.  The cG-RK discretization uses the RK2Avg integrator for $Q_1Q_0$, $Q_2Q_1$, SSP-RK3 integrator for $Q_3Q_2$ and RK4 integrator for the $Q_4Q_3$ discretization.  The cG-DeC method uses $\rm M=1$ and $\rm K=2$ for the $Q_1Q_0$, $Q_2Q_1$ discretization, $\rm M=2$ and $\rm K=5$ for the $Q_3Q_2$ discretization and $\rm M=3$ and $\rm K=10$ for the $Q_4Q_3$ discretization. The last row shows the same cG-RK method as in the top row compared with the total energy conserving cG-DeC-$\delta F$ method for $Q_1Q_0$, $Q_2Q_1$ discretizations. See the Table \ref{tab:sedov_table} for a quantitative summary of the results. In every plot, the cG-DeC methods are depicted in blue while the cG-RK are shown in red. The exact solution (solid line) is given for reference.}
    \label{fig:sedov_density}
\end{figure}

The simulation is run to $t=0.6$ on $[0,1]\times[0,1]$.  In Table \ref{tab:sedov_table}, we summarize experiments run on quadrilateral meshes using the cG-RK, cG-DeC methods for $Q_{\rm q}Q_{\rm q-1}$ for $\rm q=1,2,3,4$ as well as the cG-DeC-$\delta F$ method when using up to second order elements.  The degrees of freedom are constant across the various discretization orders.  A visual comparison of the numerical results against the semi-analytic solution is given in Figure \ref{fig:sedov_density} below.  The results indicate a trade-off between accuracy and energy conservation. The cG-RK method demonstrates both good accuracy and energy conservation properties with the RK2Avg time integrator for the $Q_1Q_0$ and $Q_2Q_1$ elements, while we see good energy conservation when the $\delta F$ correction is applied in the cG-DeC-$\delta F$ method with $\rm M = 1$, $\rm K=2$ but an increased error compared with the cG-DeC method.  For $Q_4Q_3$ elements, the cG-RK method uses the standard RK4 integrator and exhibits the expected accuracy and modest energy conservation.  The cG-DeC method matches the accuracy of the cG-RK method with slightly better accuracy for the $Q_1Q_0$ and $Q_2Q_1$ elements with $\rm M = 1$, $\rm K=2$ and better energy conservation for the $Q_4Q_3$ elements with $\rm M = 3$, $\rm K=10$ when compared with the cG-RK method with RK4 time integrator.
\begin{table}[hb!]
    \centering
    \begin{tabular}{llrrrrrrr}
    \toprule
    Order & Method & $n_{\Delta t}$ & $n_F$ & $n_{\mathcal{M}_\mathcal{V}}$ & $n_{\mathcal{M}_\mathcal{E}}$ & $n_Q$ & Total & $\|\rho-\rho_{\mathrm{ex}}\|_{L_1}$ \\
    \midrule
      \multirow{2}{*}{$Q_1Q_0$} & cG-RK (RK2Avg) & 1\,942 & 7\,872 & 193\,279 & 3\,936 & 3\,963 & 209\,050  & 1.37e-1 \\
       & cG-DeC ($\rm M=1$, $K{=}2$) & 1\,990 & 14\,063 & 8\,036 & 4\,018 & 16\,073 & 42\,190  & 1.21e-1 \\
    \midrule
      \multirow{2}{*}{$Q_2Q_1$} & cG-RK (RK2Avg) & 1\,813 & 7\,408 & 128\,344 & 3\,704 & 3\,744 & 143\,200  & 1.51e-1 \\
       & cG-DeC ($\rm M=1$, $K{=}2$) & 1\,797 & 12\,789 & 7\,308 & 3\,654 & 14\,617 & 38\,368  & 1.49e-1 \\
    \midrule
      \multirow{3}{*}{$Q_3Q_2$} & cG-RK (SSP-RK3) & 1\,282 & 7\,890 & 117\,624 & 11\,833 & 3\,979 & 141\,326  & 1.49e-1 \\
       & cG-DeC ($\rm M=2$, $K{=}5$) & 1\,274 & 36\,512 & 31\,296 & 15\,648 & 37\,817 & 121\,273  & 1.54e-1 \\
       & cG-DeC ($\rm M=2$, $K{=}3$) & 1\,277 & 20\,944 & 15\,708 & 7\,854 & 22\,254 & 66\,760  & 1.54e-1 \\
    \midrule
      \multirow{3}{*}{$Q_4Q_3$} & cG-RK (RK4) & 1\,301 & 10\,664 & 147\,174 & 15\,994 & 5\,365 & 179\,197  & 1.50e-1 \\
       & cG-DeC ($\rm M=3$, $K{=}10$) & 1\,317 & 103\,719 & 96\,984 & 48\,492 & 105\,067 & 354\,262 & 1.53e-1 \\
       & cG-DeC ($\rm M=3$, $K{=}4$) & 1\,299 & 38\,570 & 31\,920 & 15\,960 & 39\,901 & 126\,351  & 1.54e-1 \\
    \bottomrule
    \end{tabular}
    \caption{Comparison of the cG-DeC and cG-RK methods in terms of the number of time steps, force tensor applies, mass matrix applies, quadrature function evaluations and $L_1$ error of the density for the Sedov problem.  Here, we explore using the theoretical number of corrections and the number of corrections needed for convergence in the smooth test above. On this strong shock problem, we note that the error for the two choices of the number of corrections is little changed while there is a substantial gain in the efficiency of the cG-DeC method.}
    \label{tab:sedov_eff_table}
\end{table}

 In Table \ref{tab:sedov_eff_table}, we show the number of time steps taken, $n_{\Delta t}$; the number of force tensor evaluations, $n_F$; the number of kinematic mass matrix applies, $n_{\mathcal{M}_{\mathcal{V}}}$; the number of thermodynamic mass matrix applies, $n_{\mathcal{M}_{\mathcal{E}}}$; the number of quadrature function evaluations, $n_Q$; and the sum of these represented as ``Total'' similarly to Table \ref{tab:tgv_comparison} above.  The problem setup is the same as in Table \ref{tab:sedov_table}, where we keep the number of DoF fixed across the different spatial discretization orders, and the simulation is run to $t=0.6$, again. For the higher-order results we explore two choices for the number of corrections, the number of corrections needed for convergence in the smooth problem above and the theoretical number of corrections, since at most first order convergence is possible.  The results indicate that using the theoretical number of corrections increases the error very slightly, or not at all for the number of digits reported here in the $Q_3Q_2$ case, when compared to the solution obtained with the measured number of corrections needed for high-order convergence. Unsurprisingly, using the theoretical number of corrections provides a substantial gain in efficiency.  

 We remark that the cG-DeC method's efficiency relies on the number of mass matrix applies dominating the solution cost as the number of force tensor applies and quadrature function evaluations are always higher for the cG-DeC method when compared to the cG-RK method.  For more complicated equations of state, wherein the force tensor applies and quadrature function evaluations become even more substantial, the cG-DeC method may cease to be competitive with the cG-RK method.  In the present work, we focus only on the ideal gas equation of state and leave the generalization to more complicated equations of state to the future.

\subsection{Noh}\label{sec:Noh}
The Noh problem \cite{Noh.1987,Rider.WallHeating.2000} is a classic test for Lagrangian simulations.  The solution is an infinite strength shock wave.  The density is set to $\rho=1$ with the pressure $p=0$ and the velocity directed toward the origin with a magnitude of one and the adiabatic constant is $\gamma = \frac53$. The simulation domain is $[0,1]^2$ discretized with quadrilateral elements.  The simulation is run until $t=0.6$ so that the shock is located radially at $\| \bs x \| = 0.2$ and the resulting postshock density is $\rho=16$ as shown in Figure \ref{fig:Noh_den} below. 
\begin{figure}[ht!]
    \centering
    \includegraphics[width=0.6\linewidth]{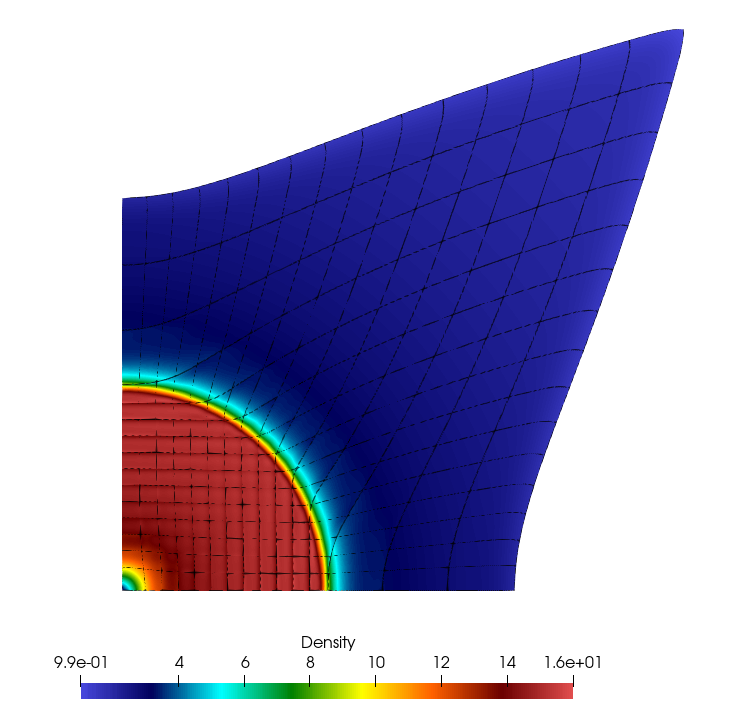}
    \caption{Density field for the Noh test problem with $Q_4Q_3$ spatial discretization using the $\rm M=3$, $\rm K=10$ cG-DeC method.}
    \label{fig:Noh_den}
\end{figure}

In Figure \ref{fig:Noh_sweep}, we show the results of a similar set of experiments to those done in the Sedov test case.  We use the same meshes and same fixed DoF count as in the Sedov tests above.  The top row of the figure shows the cG-DeC and cG-RK methods with $Q_1Q_0$, $Q_2Q_1$ spatial discretizations using the RK2Avg integrator for the cG-RK case and $\rm M=1$, $\rm K=2$ second-order DeC integrator for the cG-DeC case.  We also see the results of the $Q_3Q_2$, $Q_4Q_3$ spatial discretizations with the SSP-RK3 and RK4 integrators, respectively, for the cG-RK case and $\rm M=2$, $\rm K=5$ third-order DeC integrator and $\rm M=3$, $\rm K=10$ fourth-order DeC integrator, respectively, for the cG-DeC case.  The bottom row of the figure shows the total energy corrected cG-DeC-$\delta F$ method against the cG-RK method.  The solutions produced by the cG-DeC and the cG-RK methods are remarkably similar especially at higher-order, with the cG-DeC method producing slightly better symmetry preservation.  For the cG-DeC-$\delta F$ case, the shock location is somewhat improved for the $Q_1Q_0$ but the symmetry preservation is much worse than the uncorrected cG-DeC method for the $Q_2Q_1$ case.
\begin{figure}
    \centering
    \raisebox{2in}{\includegraphics[width=0.15\linewidth]{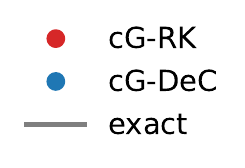}}
    \includegraphics[width=0.84\linewidth]{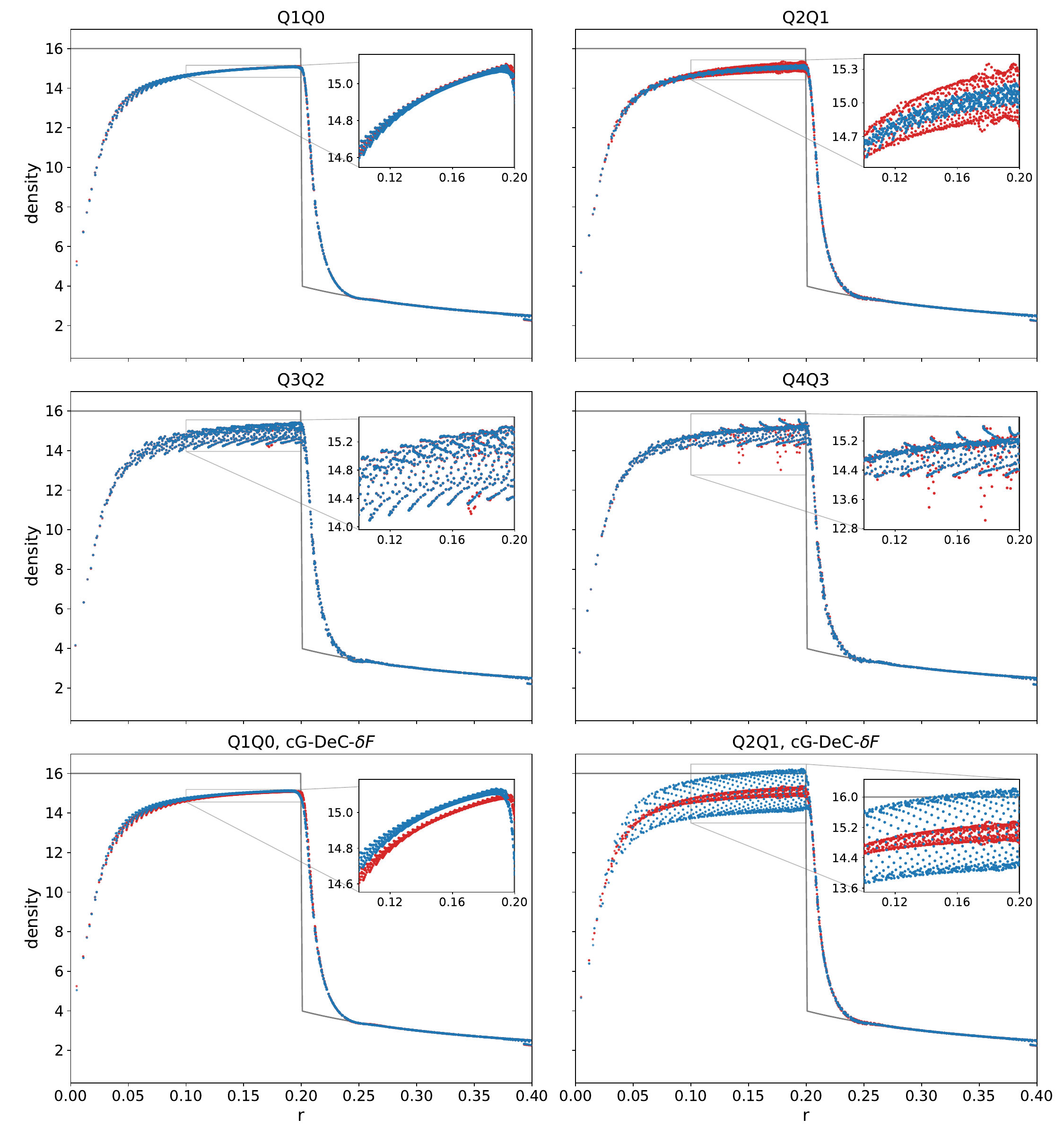}
    \caption{The computed density field for the Noh test problem across several discretization orders with the DoF held fixed across orders.  As in the Sedov test, the cG-RK method is shown in red while the cG-DeC method is shown in blue. The exact solution is provided for reference.  The top row shows the cG-DeC and cG-RK methods with $Q_1Q_0$, $Q_2Q_1$ spatial discretizations using the RK2Avg integrator for the cG-RK case and $\rm M=1$, $\rm K=2$ second-order DeC integrator for the cG-DeC case.  The middle row shows $Q_3Q_2$, $Q_4Q_3$ spatial discretizations with the SSP-RK3 and RK4 integrators, respectively, for the cG-RK case and $\rm M=2$, $\rm K=5$ third-order DeC integrator and $\rm M=3$, $\rm K=10$ fourth-order DeC integrator, respectively, for the cG-DeC case.  The bottom row shows the total energy corrected cG-DeC-$\delta F$ method against the cG-RK method.}
    \label{fig:Noh_sweep}
\end{figure}

\subsection{Triple Point}\label{sec:TrplPnt}
The triple point test problem consists of three ideal gas regions. The shock produced by the Riemann problem
induces vortical flow at the intersection of the three materials, thus testing a Lagrangian code's ability to
accurately resolve shock fronts and handle strong mesh deformations. Using the notation $\boldsymbol{x}(t,\hat{\bs x})=(x,y,z)^T$, the problem is initialized by specifying the pressure to be $p=1$ if $x\leq 1$ and $p(0,\bs{x}(0,\hat{\bs x}))=0.1$ if $x>1$, the adiabatic constant to be $\gamma=1.5$ unless $x>1$ and $y\leq 1.5$ in which case it is set to be $\gamma = 1.4$, the density is set as $\rho(0,\bs{x}(0,\hat{\bs x}))=1$ unless $x>1$ and $y \leq 1.5$ and $z \leq 1.5$ or $y>1.5$ and $z>1.5$ in which case we have $\rho(0,\bs{x}(0,\hat{\bs x}))=0.125$ in those regions.  See Table \ref{tab:3dTrplPnt_IC} for a summary of the initial conditions and Figure \ref{fig:3dTrplPnt_IC} for a depiction of the internal energy field.
\begin{table}[ht!]
    \centering
    \begin{tabular}{llll}
    \toprule
    \midrule
    $p$ & $1$ & $x \leq 1$ \\
    & $0.1$ & $x > 1$ \\
    \midrule
    $\gamma$ & $1.4$ & $x > 1,\ y \leq 1.5$ \\
    & $1.5$ & otherwise \\
    \midrule
    $\rho$ & $0.125$ & $x > 1 \text{ and } \{\ y \leq 1.5,\ z \leq 1.5\} \text{ or } \{y > 1.5,\ z > 1.5\}$ \\
    & $1$ & otherwise \\
    \bottomrule
    \end{tabular}
    \caption{Initial conditions for the Triple Point problem.}
    \label{tab:3dTrplPnt_IC}
\end{table}

\begin{figure}[ht!]
    \centering
    \includegraphics[width=0.7\linewidth]{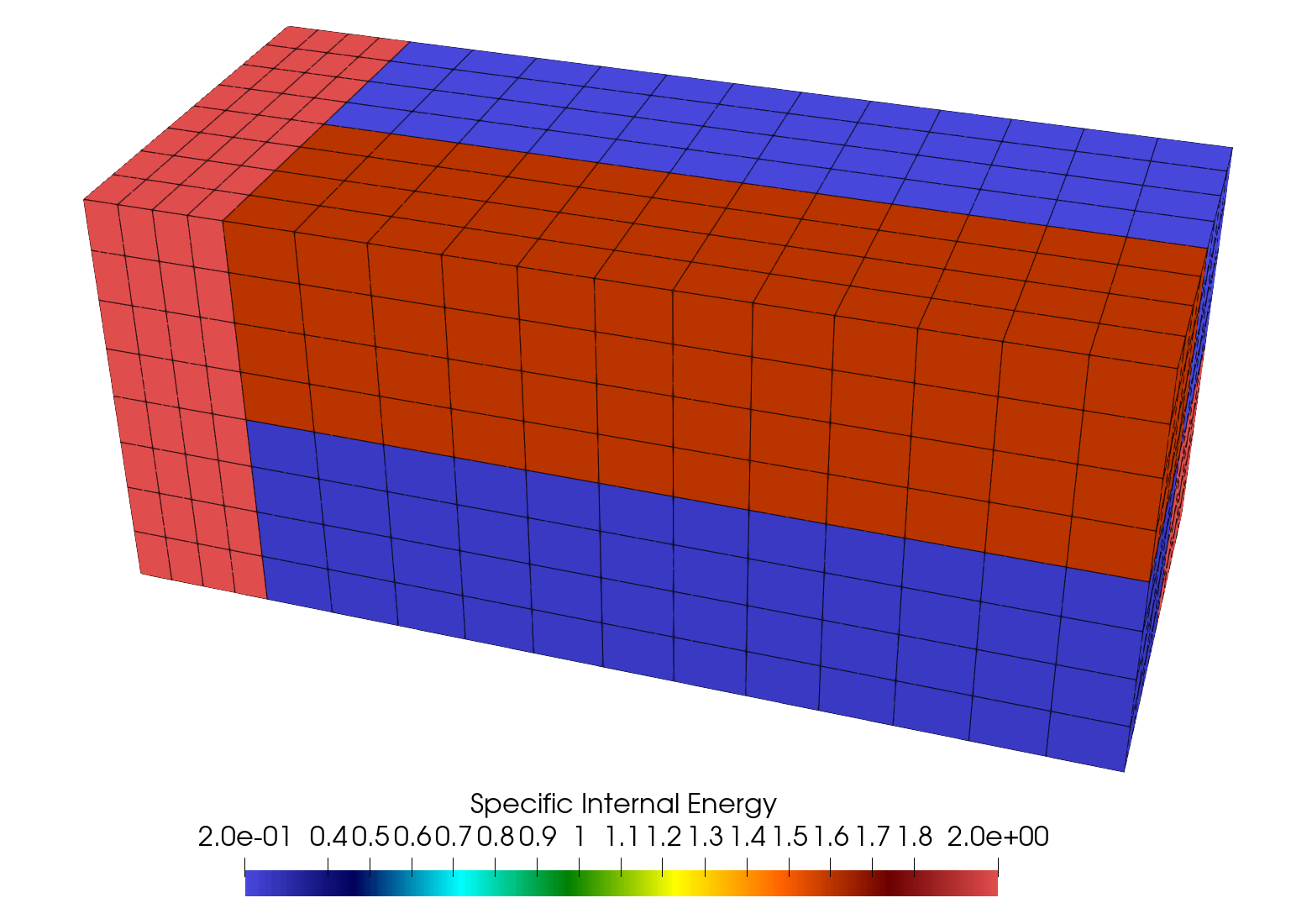}
    \caption{Initial conditions for the Triple Point problem.  The internal energy field is shown for a $Q_4Q_3$ spatial discretization.}
    \label{fig:3dTrplPnt_IC}
\end{figure}

Once again, we run the cG-RK and cG-DeC methods on the problem now only using the $Q_4Q_3$ spatial discretization with the RK4 integrator for the cG-RK method and fourth order DeC integrator with $\rm M=3$ and $\rm K=10$.  We compare the two solutions by measuring the relative $L_1$ difference using the cG-RK method as the reference solution.  Both solutions are run on a $[0,7]\times [0,3]\times[0,3]$ domain using a $16\times 8\times 8$ element hexahedral mesh until $t=3.3$.  The results are reported in Table \ref{tab:TrplPnt_rel_diff} below.
\begin{table}[hb!]
    \centering
    \begin{tabular}{llll}
    \toprule
    \midrule
    \multicolumn{2}{l}{Relative $L_1$ difference} & \multicolumn{2}{c}{$\delta TE$} \\
    \midrule
    Density & $1.16 \times 10^{-4}$ & cG-RK & $3.64 \times 10^{-4}$ \\
    Internal Energy & $9.87 \times 10^{-5}$ & cG-DeC & $1.90 \times 10^{-5}$ \\
    \bottomrule
    \end{tabular}
    \caption{Relative $L_1$ difference and total energy differences comparing the cG-DeC and cG-RK methods on the Triple Point problem.}
    \label{tab:TrplPnt_rel_diff}
\end{table}

In Figure \ref{fig:TrplPnt_Den_Sie}, the density and internal energy are shown at the final time $t=3.3$, indicating that the high-order $Q_4Q_3$ simulation effectively resolves the fields on the lower resolution mesh and that the presented cG-DeC method is robust on a complicated three-dimensional problem. The cG-RK results are visually indistinguishable from the cG-DeC and are not included.
\begin{figure}[ht!]
    \centering
    \includegraphics[width=0.7\linewidth]{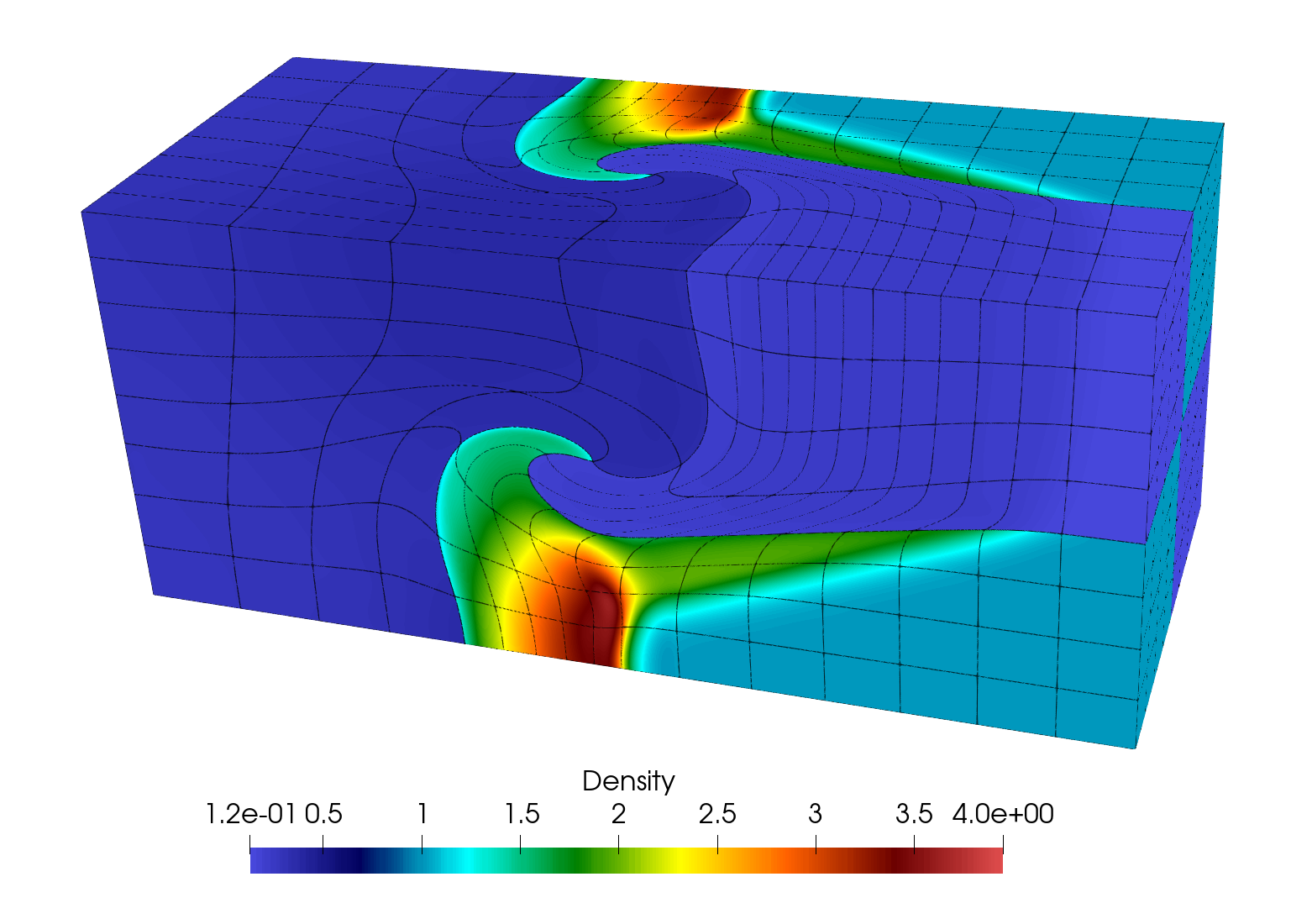}
    \includegraphics[width=0.7\linewidth]{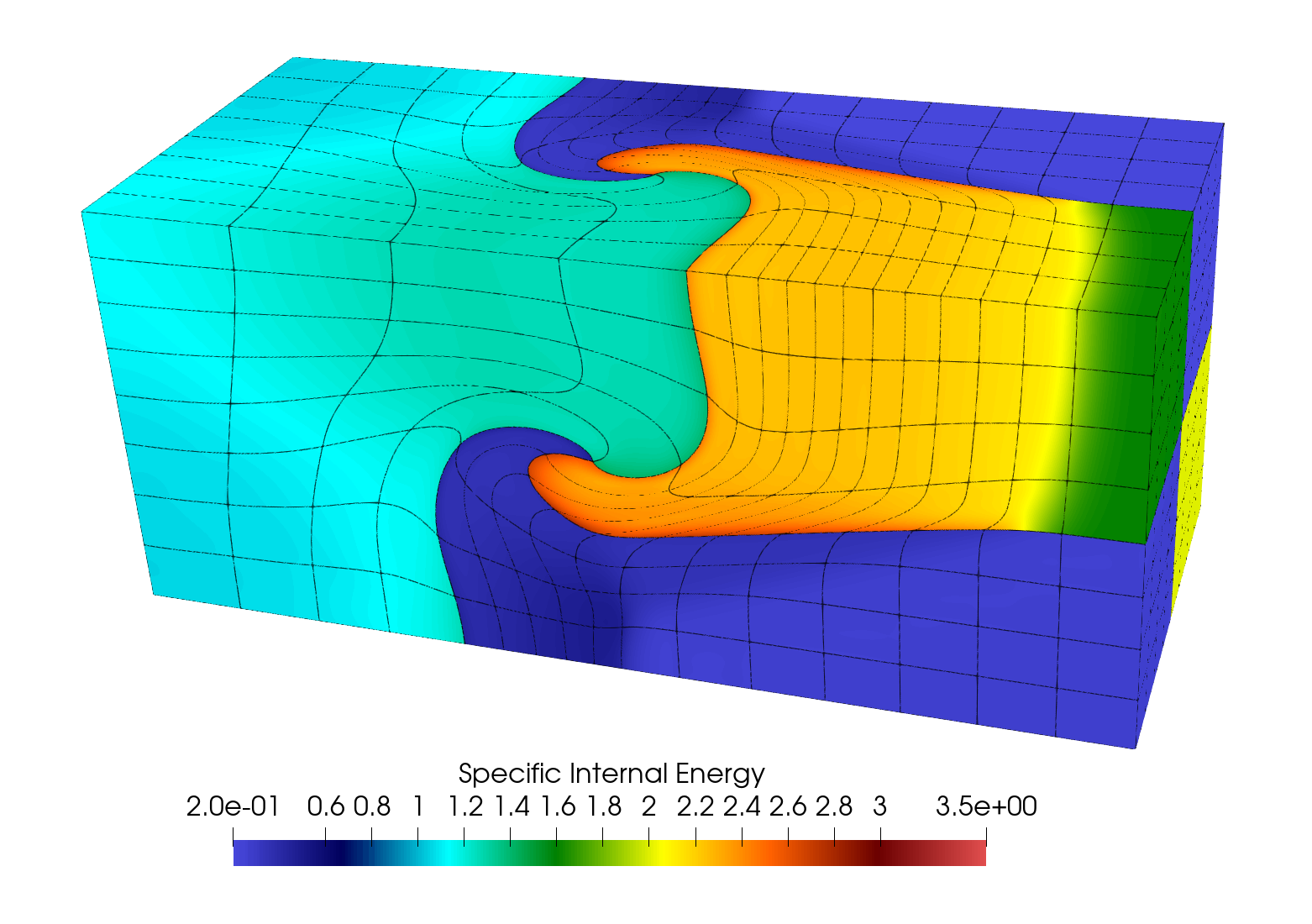}
    \caption{The density (top) and internal energy (bottom) at $t=3.3$ using the cG-DeC method with $Q_4Q_3$ elements with $\rm M=3$ and $\rm K=10$ cG-DeC discretization.}
    \label{fig:TrplPnt_Den_Sie}
\end{figure}

\section{Conclusion}\label{sec:Conclusions}
We presented a cG-DeC method for the equations of Lagrangian hydrodynamics. Total energy conservation and conservation of momentum are considered for the fully discrete scheme.  A new formulation for the DeC iteration is presented.  The method is validated on standard benchmarks and compared with the continuous Galerkin Runge-Kutta method introduced in \cite{Dobrev.2012}.  The cG-DeC method was shown to have comparable accuracy and demonstrated better efficiency on the Taylor-Green vortex problem for $Q_1Q_0$ and $Q_2Q_1$ elements, while the extra number of corrections, as measured against the theoretical number found in \cite{Abgrall.2017}, reduces the efficiency of the method for kinematic polynomial orders greater than two. Further, we compared the cG-RK and cG-DeC methods on the Sedov blast wave.  The $L_1$ errors for the cG-DeC and cG-RK methods were close to identical across polynomial orders, while the energy correction increased the measured $L_1$ errors.  The efficiency comparison was repeated, using the number of corrections needed to obtain the expected convergence rates in the previous convergence test, as well as the theoretical number of corrections.  We found that the theoretical number of corrections had little impact on the global $L_1$ errors while improving the efficiency. A similar set of tests was performed for the Noh test problem, where comparisons of the various methods are made against the exact solution.  The obtained results were found to capture the peak postshock density well and preserve symmetry.  Finally, we compared a $Q_4Q_3$ discretization of the Triple Point problem using the cG-RK and cG-DeC methods and compared the global $L_1$ difference of the solutions.  The cG-DeC method was demonstrated to provide robust solutions on this complicated problem.

The presented iteration \eqref{eqn:NeumannDeC} makes clear that the convergence in the PDE case also depends on the spectral radius of the matrix $A = I- \mathcal{D}^{-1}\mathcal{M}$, which can be close to one, though this depends on the polynomial order, dimension of the problem and basis type.  However, this suggests a path to lower the number of corrections by preconditioning the DeC operators to decrease the spectral radius of $A$.  It is also interesting to consider alternative DeC formulations, using adaptive or multilevel grids in time \cite{MicalizziTorlo2024, Ong.Spiteri.2020, Speck.2015, Speck.2016, Saupe.2025} to obtain better efficiency.

\backmatter

\bmhead{Acknowledgements}

The Los Alamos unlimited release number is LA-UR-25-29785.

\section*{Declarations}

\begin{itemize}
\item We gratefully acknowledge the support of the NNSA through the Laboratory Directed Research and Development
(LDRD) program at Los Alamos National Laboratory under project number 20230028DR. Los Alamos National Laboratory is
operated by Triad National Security, LLC for the U.S. Department of Energy’s NNSA. 
\item The authors have no relevant financial or non-financial interests to disclose.
\item The authors assert that there are no conflicts of interest, potential or otherwise.
% \item Ethics approval and consent to participate
% \item Consent for publication
% \item Data availability 
% \item Materials availability
% \item Code availability 
% \item Author contribution
\end{itemize}

% \noindent

%%===========================================================================================%%
%% If you are submitting to one of the Nature Portfolio journals, using the eJP submission   %%
%% system, please include the references within the manuscript file itself. You may do this  %%
%% by copying the reference list from your .bbl file, paste it into the main manuscript .tex %%
%% file, and delete the associated \verb+\bibliography+ commands.                            %%
%%===========================================================================================%%

\bibliography{references}% common bib file
%% if required, the content of .bbl file can be included here once bbl is generated
%%\input sn-article.bbl

\end{document}